\documentclass[11pt]{article}
\usepackage{amsmath,mathrsfs,amsfonts,amsthm,amscd,amssymb,graphicx,color,cite}
\numberwithin{equation}{section}
\usepackage{txfonts} 
\usepackage[titletoc,title]{appendix}
\usepackage{enumitem}

\newcommand{\diam}{\text{diam}}
\renewcommand{\div}{\mbox{div}\,}
\renewcommand{\div}{\mbox{div}\,}

\newcommand{\R}{{\mathbb R}} \newcommand{\K}{{\mathbb K}}

\newcommand{\e}{\varepsilon} 

\newcommand{\F}{\mathbf{F}}

\newcommand{\A}{\mathbf{A}}

\newcommand{\ba}{\mathbf{a}}

\newcommand{\M}{{\mathrm M}}
\newcommand{\calM}{{\mathcal M}}

\def\longequals{\mathbin{=\kern-2pt=}}
\def\eqdef{\mathbin{\buildrel \rm def \over \longequals}}

\newtheorem{theorem}{Theorem}[section]
\newtheorem{corollary}[theorem]{Corollary}
\newtheorem{definition}[theorem]{Definition}
\newtheorem{remark}[theorem]{Remark}
\newtheorem{lemma}[theorem]{Lemma}
\newtheorem{proposition}[theorem]{Proposition}

\numberwithin{equation}{section}

\newcommand{\beq}{\begin{equation}}
\newcommand{\eeq}{\end{equation}}

\definecolor{darkred}{rgb}{.70,.12,.20}

\definecolor{darkgreen}{rgb}{.20,.52,.14}

\title{\vspace{-1in}  Regularity estimates in weighted Morrey spaces 
for  quasilinear elliptic equations
}
\author{Giuseppe Di Fazio\footnotemark[1]
\and  Truyen  Nguyen\footnotemark[2]}

\newcommand{\Addresses}{{% additional braces for segregating \footnotesize
  \bigskip
  \footnotesize

  G.~Di Fazio, \textsc{Dipartimento di Matematica e Informatica, Universit\`a di Catania, Viale A. Doria 6, 95125 Catania, Italy.}\par\nopagebreak
  \textit{E-mail}: \texttt{giuseppedifazio@unict.it}

  \medskip

  T.~Nguyen, \textsc{Department of Mathematics,
The University of Akron, 302 Buchtel Common, Akron, OH 44325, USA.}\par\nopagebreak
  \textit{E-mail}: \texttt{tnguyen@uakron.edu}

}}

\date{}

\providecommand{\subjclass}[1]{{\textit{2010 Mathematics Subject Classification.}} #1}
\providecommand{\keywords}[1]{{\textit{Key words and phrases.}} #1}

\begin{document}
\maketitle

\begin{abstract}
We study regularity for solutions of  quasilinear  elliptic equations of the form $\div \A(x,u,\nabla u) = \div \F $ in 
bounded domains in $\R^n$. The vector field $\A$ is assumed to be continuous in $u$, and its growth in $\nabla u$ is like that of 
the $p$-Laplace operator. We establish  interior gradient estimates in weighted Morrey spaces  for  weak solutions $u$ to the 
equation under a small BMO condition in $x$ for $\A$.  As a  consequence, we obtain
that $\nabla u$ is   in the classical Morrey space 
$\calM^{q,\lambda}$ or weighted space $L^q_w$ whenever  $|\F|^{\frac{1}{p-1}}$ is  respectively  in $\calM^{q,\lambda}$ or 
$L^q_w$, where $q$ is any number greater than $p$ and $w$ is any weight in the   Muckenhoupt class $A_{\frac{q}{p}}$. 
In addition, our  two-weight estimate allows  the possibility  to acquire the regularity for $\nabla u$ in a 
weighted Morrey space that is different from the functional space that the data $|\F|^{\frac{1}{p-1}}$ belongs to.
\end{abstract}
\\

\renewcommand{\thefootnote}{\fnsymbol{footnote}}

%\footnotetext[1]{Dipartimento di Matematica e Informatica, Universit\`a di Catania, Viale A. Doria 6, 95125 Catania, Italy.}

\footnotetext[2]{The research of the  author is supported in part
by a grant  from the Simons Foundation (\# 318995) and by the UA Faculty Research grant FRG
(\# 207359).} 

%\footnotetext[1]{Research of the authors was supported in part by.}

%\parindent=0pt

%\tableofcontents

%\renewcommand{\thefootnote}{\fnsymbol{footnote}}

%\footnotetext[1]{Department of Mathematics,
%The University of Akron, Akron, OH~ 44325, USA. Email: tnguyen@uakron.edu.
%The research of the  author is supported in part
%by a grant  from the Simons Foundation (\# 318995) and by the UA Faculty Research grant FRG
%(\# 207359)}

%\tableofcontents
\setcounter{equation}{0}

\section{Introduction}\label{sec:Intro}
We investigate interior gradient estimates in weighted Morrey space  for bounded weak solutions to  the following general elliptic equations of $p$-Laplacian type
\begin{equation}\label{GE}
\div \A(x, u, \nabla u)=  \div \F  \quad \mbox{in}\quad \Omega
\end{equation}
when $\Omega$ is a bounded domain in $\R^n$ ($n\geq 1$) and the vector field $\A$ is only continuous in the $u$ variable and possibly 
discontinuous in the $x$ variable. Without loss of generality we take $\Omega$ to be the Euclidean ball 
$B_{10}  := \{x\in \R^n:\, | x| <10\}$, whose special radius is simply to avoid working with many fractions in our arguments later on.  
Let  $\K\subset \R$ be an open interval and consider the general vector field
\[
\A = \A(x,z,\xi) : B_{10}\times \overline\K\times \R^n \longrightarrow \R^n
\]
which is  a Carath\'eodory map, that is, $\A(x,z,\xi)$ is measurable in $x$ for every $(z,\xi)\in \overline \K \times \R^n$ and 
continuous in $(z,\xi)$ for a.e. $x$. We assume that $\xi \mapsto\A(x,z,\xi)$ is differentiable on 
$ \mathbb{R}^n\setminus \{0\}$ for a.e. $x$ and all $z\in \overline \K$. Also,  there exist constants  $\Lambda>0$,  $1< p<\infty$, 
and a nondecreasing and right continuous function $\omega: [0,\infty)\to [0,\infty)$ with $\omega(0)=0$
such that  the following  conditions are satisfied  for a.e. $x\in B_{10}$ and all $z\in \overline\K$:
\begin{align}
&\langle \partial_\xi \A(x,z,\xi) \eta, \eta\rangle \geq \Lambda^{-1} \, |\xi|^{p-2} |\eta|^2\,  \qquad\qquad \qquad 
\qquad\forall \xi\in \R^n\setminus \{0\} \mbox{ and }\forall \eta\in\R^n,\label{structural-reference-1}\\ 
& |\A(x,z,\xi)|  + |\xi|\,  | \partial_\xi \A(x,z,\xi) | \leq \Lambda |\xi|^{p-1}\qquad\qquad\qquad \forall\xi\in  \R^n,\label{structural-reference-2}\\
& |\A(x,z_1,\xi)-\A(x,z_2,\xi)|  \leq \Lambda |\xi|^{p-1} \omega(|z_1 - z_2|) \qquad\quad \,\forall z_1, z_2\in \overline\K\mbox{ and }
\forall \xi\in \R^n. \label{structural-reference-3}
\end{align}

 The class of equations of the form \eqref{GE} with $\A$ 
 satisfying \eqref{structural-reference-1}--\eqref{structural-reference-3} contains  the  well known $p$-Laplace equations and, more generally,  equations of the type
  \begin{equation}\label{like-p-Laplacian}
\div \A(x,  \nabla u)=  \div \F  \quad \mbox{in}\quad \Omega.
\end{equation}
Gradient estimates for \eqref{like-p-Laplacian} with discontinuous coefficient 
have  been studied extensively  
by several authors
 \cite{BR,CP,DM,Di,DM2,DM3,Gi,I,KZ,LU,MP1,MP2,Ra,Uh}.  Some of these developments rely essentially  on the  perturbation method developed by 
 Caffarelli and Peral \cite{CP}, which   allows ones to deal with discontinuous coefficient and highly nonlinear structure in gradient of equation \eqref{like-p-Laplacian}.
 
In this paper we study  general quasilinear equation
\eqref{GE} when the principal part also depends on the $z$ variable.  In the case $\A$ is  Lipschitz continuous in both $x$ and $z$ variables, the interior $C^{1,\alpha}$ regularity for locally bounded weak solutions to the 
corresponding homogeneous equation was established by DiBenedetto \cite{D} and Tolksdorf \cite{T}  extending the celebrated  $C^{1,\alpha}$ estimates by   
Ural\'tceva \cite{Ur}, Uhlenbeck \cite{Uh}, Evans \cite{E1}, and Lewis \cite{Le}
for the homogeneous $p$-Laplace equation. When $\A$ is   not necessarily continuous in $x$ but has sufficiently small BMO oscillation, 
it was established in  \cite{NP} that: $|\F|^{\frac{1}{p-1}}\in L^q_{loc}  \Longrightarrow \nabla u\in L^q_{loc}$ for any $q>p$. 
It is required 
in  \cite{NP}  that $\A$ is Lipschitz continuous in the $z$ variable
since  the uniqueness property  for the frozen equation is used  there. However, this condition is 
weaken  in a recent paper \cite{BPS}  allowing only H\"older continuity in $z$ for $\A$ (i.e. $\omega(s) =s^\alpha$ 
 for some $\alpha\in (0,1)$ in \eqref{structural-reference-3}).
%\marginpar{Use $V=V(\nabla u)=|\nabla u|^{p-2} \nabla u$}

Our purpose of the current work is to 
 extend the mentioned  result in \cite{NP,BPS} by 
 deriving interior  estimates in weighted Morrey spaces for gradients of  locally bounded weak solutions to nonhomogeneous  equation \eqref{GE}.  
 %Morrey and weighted $L^q$ estimates 
  In order to state our main results, let us recall the so-called Muckenhoupt class of $A_s$ weights. By definition,  a weight is  a nonnegative  locally  integrable function on $\R^n$ that is positive almost everywhere.
A weight $w$ belongs to the  class $A_s$, $1<s< \infty$, if 
\[
 [w]_{A_s} :=  \sup{ \Big(\fint_{B} w(x)\, dx\Big) \Big(\fint_{B} w(x)^{\frac{-1}{s-1}}\, dx\Big)^{s-1}  } <\infty,
\]
where the supremum is  taken over all balls $B$ in $\R^n$.  
We also say that  $w$ belongs to the class $A_\infty$  if 
\[
 [w]_{A_\infty}  :=  \sup{ \Big(\fint_{B} w(x)\, dx\Big) \exp{\Big(\fint_{B} \log{ w(x)^{-1}}\, dx\Big)}  } <\infty.
\]
%Here the supremums are taken over all balls $B$ in $\R^n$.  
Now let $U\subset \R^n$ be a bounded open set, $w$ be a weight, $1\leq q <\infty$, and $\varphi$ be a positive function on the set of nonempty
open balls in $\R^n$. 
%We recall that 
A function $g: U\to \R$ is said to 
belong to the weighted space $L^q_w(U)$ if
\[
 \|g\|_{L^q_w(U)} := \Big(\int_{U} |g(x)|^q w(x)\, dx\Big)^{\frac1q}<\infty.
\]
%We note that $\Big( L^p_w(U), \|\cdot \|_{L^q_w(U)}\Big) $ is a Banach space and it follows from H\"older inequality that $ L^{p s}_w(U) \subset  L^p(U)$ whenever $w\in A_s$ with $1<s<\infty$. 
We define the weighted Morrey 
space $\calM^{q,\varphi}_w(U)$  to be the set of all functions $g\in L^q_w(U)$ satisfying
\begin{equation}\label{weighted-morrey}
\|g\|_{\calM^{q,\varphi}_w(U)} := \sup_{ \bar x\in U,\, 0<r\leq \diam(U)} \Bigg(\frac{\varphi(B_r(\bar x))}{ w(B_r(\bar x))} \int_{B_r(\bar x)\cap U}
|g(x)|^q w \, dx \Bigg)^{\frac1q}<\infty.
\end{equation}
Notice that $\calM^{q,\varphi}_w(U)=L^q_w(U)$ if  $\varphi(B)=  w(B)$, and  we obtain the classical 
Morrey space $\calM^{q,\lambda}(U)$ ($0\leq \lambda\leq n$) by taking $w=1$ and $\varphi(B)= |B|^{\frac{\lambda}{n}}$. It is worth observing that 
$\calM^{q,\lambda}(U)= L^q(U)$ when $\lambda=n$ and  $\calM^{q,\lambda}(U)= L^\infty(U)$ when $\lambda=0$.
For brevity,
the Morrey 
space $\calM^{q,\varphi}_w(U)$ with $w=1$ will be  denoted by 
$\calM^{q,\varphi}(U)$.  Let us next introduce some slight restrictions on the function  $\varphi$.
\begin{definition}\label{class-B} Let $\varphi$ be a positive function on the set of nonempty
open balls in $\R^n$. We say that $\varphi$ belongs to the class $\mathcal B_\alpha$ with $\alpha\geq 0$ if there exists  $C>0$ such that
\[
  \frac{\varphi(B_r(x))}{\varphi(B_s(x))}\leq C \Big(\frac{r}{s}\Big)^\alpha
 %\quad \mbox{ and }\quad \Phi(B_r(x))\leq C r^\sigma 
\]
for every $x\in\R^n$ and every $0<r\leq s$. We define $\mathcal B_+ :=\cup_{\alpha>0} \mathcal B_\alpha \subset \mathcal B_0$.
 \end{definition}
 % \enlargethispage{\baselineskip}
 % \pagebreak 
 %\noindent 
 The function $\varphi(B) := |B|^{\frac{\alpha}{n}}$ with $\alpha>0$ is an  example of functions in $\mathcal B_+$. Also for any weight 
 $w\in A_\infty$, the function   $\varphi(B) := w(B)$  belongs to the class $\mathcal B_+$ due to the characterization \eqref{charac-A-infty} below. However, a member 
of $\mathcal B_+$ does not need to be even a measure. Hereafter, for a 
ball $B\subset \R^n$ we set
$\bar\A_{B}(z,\xi) := \fint_{B} \A(x,z,\xi)\, dx$. Also, the conjugate exponent of a number $l\in (1,\infty)$ is denoted by 
$l'$.
Our first main result is:
\begin{theorem}\label{thm:weighted-morrey} %Let 
%$\K\subset \R$ be an open interval and 
Let $\A$
%: B_{10}\times \overline\K\times \R^n \to %\R^n$ 
satisfy 
\eqref{structural-reference-1}--\eqref{structural-reference-3} with $p>1$,  and let $w$ be an $A_s$ weight for some $1<s<\infty$.
Then for any $q\geq p$,  $M_0>0$, and $\varphi\in \mathcal B_+$ with $\sup_{x\in B_1} \varphi( B_2(x))<\infty$, 
there exists a constant   $\delta=\delta( p, q,  n,\omega,  \Lambda, M_0,s, [w]_{A_s})>0$  such that: if
\begin{equation}\label{smallness-1}
\sup_{0<\rho\leq 1}\sup_{y\in B_{\frac{15}{2}}} \sup_{z\in \overline\K \cap [-M_0, M_0]} \fint_{B_\rho(y)} \Big[
\sup_{\xi\neq 0}\frac{|\A(x,z,\xi) - \bar\A_{B_\rho(y)}(z,\xi)|}{|\xi|^{p-1}}
\Big] \,dx \leq \delta,
\end{equation}
and $u$  is a weak solution of 
\begin{equation}\label{GE-10}
\div \A(x, u, \nabla u)=  \div \F\quad \mbox{in}\quad B_{10}
\end{equation}
satisfying $\|u\|_{L^\infty(B_{10})}\leq M_0$, we have
\begin{equation}\label{weighted-morrey-est}
\| \nabla u \|_{\calM^{q,\varphi}_w(B_1)} \leq    C\Bigg(
\|\nabla u\|_{L^p(B_{10})} +  \| \M_{B_{10}}(|\F|^{\frac{p}{p-1}})\|_{\calM^{1, \varphi^{\frac{p}{q}}}(B_{10})}^{\frac{1}{p}} 
+ \|\M_{B_{10}} (|\F|^{\frac{p}{p-1}})\|_{\calM^{\frac{q}{p},\varphi}_w(B_{10})}^{\frac{1}{p}}\Bigg).
\end{equation}
Here $\M_{B_{10}}$ denotes  the centered Hardy--Littlewood   maximal operator (see Definition~\ref{def:maximal}), and  $C>0$ 
is a constant depending only on  $q$, $p$, $n$,  $\omega$,  $\Lambda$, $M_0$, $\varphi$, $s$,  and $[w]_{A_s}$.
\end{theorem}
The above theorem holds true for any weight $w$ in the class $A_\infty$. When $q>p$ and certain  additional information about the weights and $ \varphi, \, \phi$ is 
given, we can further estimate the two quantities in \eqref{weighted-morrey-est} involving the maximal function
of $|\F|^{\frac{p}{p-1}}$  to obtain:

\begin{theorem}[weighted Morrey space estimate]\label{cor:weighted-morrey} Let 
%$\K\subset \R$ be an open interval and  
$\A$
% B_{10}\times \overline\K\times \R^n \to 
%\R^n$ 
satisfy 
\eqref{structural-reference-1}--\eqref{structural-reference-3} with $p>1$.
% Let $\phi$ be a positive function on the set of open balls in $\R^n$.
 Let  $q>p$, $w\in A_\infty$, $v\in A_{\frac{q}{p}}$, $\varphi\in \mathcal B_+$ with $\sup_{x\in B_1} \varphi( B_2(x))<\infty$,
and $\phi\in \mathcal B_0$  satisfy
\begin{align}
 & [w, v^{1-(\frac{q}{p})'}]_{A_{\frac{q}{p}}}:= \sup_{B}{ \Big(\fint_{B} w\, dx\Big) \Big(\fint_{B} v^{1-(\frac{q}{p})'}\, dx\Big)^{\frac{q}{p}-1}  } 
  <\infty,\label{cond:wv}\\
 & \frac{v(2 B)}{w(2 B)} \frac{1}{\phi(2 B)} \leq C_* \, \frac{1}{\varphi(B)} 
\quad \mbox{for all balls  $B\subset \R^n$}.\label{cond:vwvarphi} 
 \end{align}
 %Assume in addition that either $w\in %A_{\frac{q}{p}}$ or $v\in A_{\frac{q}{p}}$.
Then for any $M_0>0$, there exists a small constant 
$\delta=\delta( p, q,  n, \omega, \Lambda, M_0, [w]_{A_\infty})>0$  such that: if
\eqref{smallness-1} holds 
%\begin{equation*}
%\sup_{0<\rho\leq 1}\sup_{y\in B_{\frac{15}{2}}} \sup_{z\in \overline\K \cap [-M_0, M_0]} \fint_{B_\rho(y)} \Big[
%\sup_{\xi\neq 0}\frac{|\A(x,z,\xi) - \A_{B_\rho(y)}(z,\xi)|}{|\xi|^{p-1}}
%\Big] \,dx \leq \delta,
%\end{equation*}
and $u$  is a weak solution of \eqref{GE-10}
%\[\div \A(x, u, \nabla u)=  \div \F\quad %\mbox{in}\quad B_{10}
%\]
satisfying $\|u\|_{L^\infty(B_{10})}\leq M_0$, we have
\begin{equation}\label{weighted-morrey-final}
\| \nabla u \|_{\calM^{q,\varphi}_w(B_1)}  \leq    C\left( \|\nabla u\|_{L^p(B_{10})}+  
\| |\F|^{\frac{1}{p-1}}\|_{\calM^{q,\phi}_v(B_{10})} \right).
\end{equation}
Here  $C >0$ is a  constant 
 depending only on  $q$, $p$, $n$, $\omega$,  $\Lambda$,  $M_0$,  $\varphi$, $\phi$, $C_*$,  $[w]_{A_\infty}$,  $[v]_{A_{\frac{q}{p}}}$, and  
 $[w, v^{1-(\frac{q}{p})'}]_{A_{\frac{q}{p}}}$.
 %either on $[w]_{A_{\frac{q}{p}}}$ or on 
\end{theorem}
\begin{remark}
 If $\phi(2B)\leq C \varphi(B)$ for all balls $B$, then  conditions \eqref{cond:wv}--\eqref{cond:vwvarphi}  imply that $v(x)\leq C C_* w(x)$ a.e. and 
   $v\in A_{\frac{q}{p}}$ with 
  $[v]_{A_{\frac{q}{p}}} \leq C C_* 
 [w, v^{1-(\frac{q}{p})'}]_{A_{\frac{q}{p}}}$.
\end{remark}
%\enlargethispage{\baselineskip}
%\pagebreak 

Estimate \eqref{weighted-morrey-final} is a two-weight inequality which allows  the possibility  to acquire the regularity for $\nabla u$ in a 
weighted Morrey space that is different from the functional space that the data $|\F|^{\frac{1}{p-1}}$ belongs to.
In particular, the result  in Theorem~\ref{cor:weighted-morrey}  is even new when applied to equations of special form \eqref{like-p-Laplacian}
whose gradient estimates in the classical  Morrey spaces  are obtained in  \cite{Ra,MP1,MP2}.
%For gradient estimates in 
%standard $L^q$ spaces for  %general equation \eqref{GE}, 
Our results also improve the $L^q$ gradient estimates established in \cite{NP,BPS} for equation \eqref{GE}
since the principal part $\A(x,z,\xi)$ is merely assumed to be   continuous in the $z$ variable instead of being H\"older continuous. 

One of the main difficulties in proving the above main results is 
that equations of the  form  \eqref{GE-10} are not  invariant with respect to dilations and rescaling of domains due to the dependence of $\A(x,z,\xi)$ on the $z$ variable. To handle this
issue, we use the key idea introduced in \cite{HNP1,NP} by  enlarging  the class of equations under consideration in a suitable way. Precisely, 
we  consider the associated quasilinear elliptic  equations with two parameters, i.e.  equation \eqref{ME} below.
The class of these  equations is the smallest one that is invariant with respect to dilations and rescaling 
of domains and that contains equations of the form \eqref{GE-10}.
Given the invariant structure, we employ the direct argument 
in \cite{AM,BPS} to show that  the gradient of the solution $u$ can be approximated
by a bounded gradient in $L^p$ norm (see
Proposition~\ref{lm:localized-compare-gradient}). It is essential for us  that  
this approximation is 
uniform with respect to the two parameters. With this and through a standard procedure, we derive some density estimates 
and obtain Theorem~\ref{thm:main} about gradient estimates in weighted $L^q$ spaces (see Subsections~\ref{sub:density-est} and \ref{sub:Lebesgue-Spaces}). This estimate plays a central role in 
proving our main results. Indeed, by using 
 the trick in \cite{MP2} we show in Subsection~\ref{sec:Morrey-Spaces} that
Theorem~\ref{thm:weighted-morrey} can be derived as  a consequence of Theorem~\ref{thm:main}.
In order to prove Theorem~\ref{cor:weighted-morrey}, we need to further estimate the last two terms  in inequality \eqref{weighted-morrey-est} involving the maximal function
of $|\F|^{p'}$.
This is another difficulty that needs to overcome and it can be solved  if one has 
a suitable estimate for the maximal function in weighted Morrey spaces. This type of estimate is proved 
by Chiarenza and Frasca  \cite{ChF} in the unweighted setting, and there are some recent works 
\cite{GMOS,GORS,KS,Na}  establishing  the estimate in the weighted setting under certain conditions. However, they are inadequate for our purpose
and we need to extend these results by attaining  in Theorem~\ref{thm:maximal-est} a weighted Morrey space estimate 
for the maximal function.  This two-weight 
inequality holds true  for quite general weights and we believe that it is of independent interest.
% under quite general conditions on the %weights.

We end the introduction by pointing out that our method of establishing   interior gradient estimates in weighted Morrey spaces could be 
combined with the boundary techniques used in \cite{BR,BPS,MP1,MP2} to derive global estimates for Reifenberg flat  domains and homogeneous 
Dirichlet boundary condition. Furthermore, 
by following \cite{Ph} 
it might be  possible to weaken 
the assumption $u\in L^\infty $ in Theorem~\ref{cor:weighted-morrey} by assuming only that $u\in BMO$.
%Based on the technique introduced in the recent paper  \cite{N}, we anticipate that 
%the result in Theorem~\ref{cor:weighted-morrey} can be extended to general parabolic equations of $p$-Laplacian type. 
 %We leave these for interested readers. 

\bigskip

The organization of the paper is as follows. We recall some basic properties of 
$A_s$ weights in Subsection~\ref{sub:weight} and state a key result (Theorem~\ref{thm:main}) in Subsection~\ref{sub:eqparameters}
about gradient estimates in weighted $L^q$ spaces. In Section~\ref{sub:two-weight-est}, we derive a weighted Morrey space estimate 
for the maximal function (see  Theorem~\ref{thm:maximal-est} and its corollaries).  Section~\ref{approximation-gradient} is devoted to proving
Proposition~\ref{lm:localized-compare-gradient}
which shows  that gradients of weak solutions to  two-parameter equation  \eqref{ME} can be approximated by bounded gradients under some smallness conditions. Using this crucial result,
we establish  in Section~\ref{interior-density-gradient} some density estimates 
for gradients and then prove Theorem~\ref{thm:main}. In this same section,  the main results stated in Theorem~\ref{thm:weighted-morrey}
 and Theorem~\ref{cor:weighted-morrey} are derived as consequences of Theorem~\ref{thm:main}.
 
\section{Preliminaries and a key result}\label{sec:preliminary}
\subsection{Some basic properties of $A_s$ weights}\label{sub:weight}
%The function 
%$|x|^a$ is an $A_s$ weight, $1<s<%\infty$,  if and only if $-n<a< %n(s-1)$;  
%$|x|^a$ is an $A_\infty$ weight  if %and only if $a>-n$. It is also well %known that if $P(x)$ is a %polynomial of degree $k$ in $\R^n%$, then $|P(x)|$ is an $A_\infty$ %weight. 
Let us recall some  properties of weights which can be found in \cite[Chapter~9]{Gra}. Given a weight $w$ and a measurable set $E\subset \R^n$, we use the notation $dw(x)= w(x)\, dx$ and  $w(E)=\int_E w(x)\, dx$.

\begin{lemma}\label{weight:basic-pro} 
Let $w\in A_s$ for some $1< s <\infty$. Then
\begin{enumerate}
%\item[1)]  The measure $w(x)\, dx$ is doubling, that is, there exists  $C>0$ such that $w(2B)\leq C w(B)$ for all balls $B$.
\item[1)] There exist $0<\beta \leq 1$ and $K>0$ depending only on $n$ and $[w]_{A_s}$ such that 
\begin{equation}\label{strong-doubling}
[w]_{A_s}^{-1}\, \Big(\frac{|E|}{|B|}\Big)^s \leq  \frac{w(E)}{w(B)}\leq K\,  \Big(\frac{|E|}{|B|}\Big)^\beta
\end{equation}
  for all balls $B$ and all measurable sets $E\subset  B$. In particular, $w$ is doubling with
   $w(2B)\leq 2^{ns} [w]_{A_s}  w(B)$.
   %for all balls $B$.
  \item[2)] The function $w^{1-s'}$ is in $A_{s'}$ with characteristic constant 
$\, [w^{1-s'}]_{A_{s'}}=[w]_{A_s}^{\frac{1}{s-1}}.
$
 %In particular, $w\in A_2$ if and only if  $w^{-1}\in A_2$.
  \item[3)] The classes $A_s$ are increasing as $s$ increases; precisely, for $1< s<q\leq \infty$ we have $[w]_{A_q}\leq [w]_{A_s}$. 
  \item[4)] For any $1<q\leq \infty$, we have
  $\,\, A_q= \displaystyle\bigcup_{1< s<q} A_s$.
 \end{enumerate}
\end{lemma}
  
 \begin{lemma}[Characterizations of $A_\infty$ weights] Suppose that $w$ is a weight. Then $w$ is in $A_\infty$ 
 if and only if there exist $0<A, \, \nu<\infty$ such that  for all balls $B$ and all measurable sets $E\subset  B$  we have
\begin{equation}\label{charac-A-infty}
\frac{w(E)}{w(B)}\leq A \Big(\frac{|E|}{|B|}\Big)^{\nu}.
\end{equation}
When $w\in   A_\infty$,  the above constants $A$ and $\nu$ depend only on $n$ and  $[w]_{A_\infty}$. Conversely, given constants $A$ and $\nu$  satisfying  \eqref{charac-A-infty},  we have 
$[w]_{A_\infty}\leq C(n,A,\nu)$.
\end{lemma}

\subsection{Quasilinear  equations with two parameters}
\label{sub:eqparameters}
Our goal is to derive interior gradient estimates for weak solutions to \eqref{GE-10}.
%\begin{equation}\label{ME-1}
%\div \A(x, u, \nabla u)=  \div \F \quad \mbox{in}\quad B_{10}.
%\end{equation}
Let us consider a function  $u\in W_{loc}^{1,p}(B_{r R})$  such that $u(y)\in \overline\K$ for a.e. $y\in B_{r R}$ and
$u$ satisfies 
$
\div \A(y,u,\nabla u) =  \div \F$ in $B_{r R }$
in the  sense of distribution. Then the rescaled  function 
\begin{equation}\label{scaled-solutions}
 v(x) := \frac{u(r x)}{\mu\,  r} \quad \mbox{for}\quad r, \, \mu>0
 \end{equation}
has the properties: $v(x)\in \frac{1}{\mu r}\overline\K$ for a.e. $x\in B_R$ and $v$ solves the equation
\begin{equation*}
\div \A_{\mu, r}(x,\mu r v,\nabla v) =  \div \F_{\mu, r}\quad \text{in}\quad B_R
\end{equation*}
in the distributional sense. Here,
\begin{equation*} 
\A_{\mu, r}(x,z,\xi) := \frac{\A(rx,  z, \mu \xi)}{\mu^{p-1}} \quad\mbox{ and }\quad  \F_{\mu, r}(x) := \frac{\F(r x)}{\mu^{p-1}}. 
 \end{equation*} 
 It is clear that if  $\A: B_{r R}\times \overline\K\times \R^n \longrightarrow \R^n$ satisfies  conditions \eqref{structural-reference-1}--\eqref{structural-reference-3}, then 
the rescaled vector field $ \A_{\mu, r}: B_{R}\times \overline\K\times \R^n \longrightarrow \R^n$ also satisfies these conditions 
with the same constants. 

The above observation  shows that  equations of type \eqref{GE-10} are not invariant with respect to the standard scalings 
\eqref{scaled-solutions}. This presents a serious obstacle in obtaining weighted $L^q$ estimates for their solutions  
by using the methods in \cite{BR,MP1,MP2}. 
Here we follow the idea in \cite{HNP1,NP}  by considering associated quasilinear equations with two parameters 
\begin{equation}\label{ME}
\div \Big[ \frac{\A(x, \lambda \theta u, \lambda \nabla u)}{\lambda^{p-1}}\Big]=  \div \F \quad \mbox{in }\, B_{10}
\end{equation}
 with $\lambda,\, \theta>0$. The class of these equations is the smallest one that is invariant with respect to the 
 transformations \eqref{scaled-solutions}
and that contains equations of type \eqref{GE-10}. Indeed,  if $u$ solves \eqref{ME} and $v$ is given by  \eqref{scaled-solutions}, then $v$ 
 satisfies an equation of similar form, namely,  $
\div \big[ \frac{\A'(y, \lambda' \theta' v, \lambda' \nabla v)}{\lambda'^{p-1}}\big]=  \div \F'$ in  $B_{\frac{10}{r}}$
with $\A'(y, z, \xi) := \A(r y, z, \xi)$, $\F'(y) := \F(r y )/\mu^{p-1}$, $\lambda' := \mu \lambda$,  and $\theta':=r\theta$.

Let us give the precise definition of weak solutions that is used throughout the paper.
\begin{definition}
\label{weak-sol}
Let $\F\in L^{p'}(B_{10};\R^n)$. A function $u\in W^{1,p}_{\text{loc}}(B_{10})$ is called a 
weak solution  of \eqref{ME} if  $u(x)\in \frac{1}{\lambda \theta}\overline\K \,$ for a.e. $x\in B_{10}$ and 
\begin{equation*}
\int_{B_{10}} \Big\langle\frac{ \A (x,\lambda \theta u,\lambda \nabla u)}{\lambda^{p-1}}, \nabla \varphi \Big\rangle\, dx 
=\int_{B_{10}} \langle \F, \nabla \varphi\rangle \, dx\qquad \forall \varphi \in W_0^{1,p}(B_{10}).
 \end{equation*}
\end{definition}
Fix a number $M_0>0$. Then for a ball $B\subset \R^n$, we define 
\begin{equation}\label{def:Theta}
\Theta_{B}(\A):= \sup_{z\in \overline\K \cap [-M_0, M_0]} \fint_{B}\Big[\sup_{ \xi\neq 0}\frac{|\A(x,z,\xi) - \bar\A_{B}(z,\xi)|}{|\xi|^{p-1}}\Big]\, dx.
\end{equation}
In order to achieve the main results stated in Section~\ref{sec:Intro}, we will prove the following gradient estimate in 
weighted $L^q$ spaces:

\begin{theorem}\label{thm:main}
 Let %$\K\subset \R$ be an open interval 
 $\A$ satisfy \eqref{structural-reference-1}--\eqref{structural-reference-3} with $p>1$, 
 and let $w$ be an $A_s$ weight for some $1<s<\infty$. 
For any $q\geq p$ and $M_0>0$,  there exists 
a  constant  $\delta=\delta( p, q,  n, \omega, \Lambda, M_0,s, [w]_{A_s})>0$    such that:  if $\lambda>0$, $\theta>0$,
\begin{equation}\label{ABCD}
\sup_{0<\rho\leq \frac15}\sup_{y\in B_\frac32} \Theta_{B_\rho(y)}(\A)\leq  \delta,
\end{equation}
and  $u$ is a weak solution of \eqref{ME} satisfying 
 $\|u\|_{L^\infty(B_2)}\leq \frac{M_0}{\lambda \theta}$, then
 \begin{equation}\label{main-estimate}
\fint_{B_1} |\nabla u|^q\, dw \leq  C\Bigg(
\|\nabla u\|_{L^p(B_2)}^q + \fint_{B_\frac32}\M_{B_2}(|\F|^{\frac{p}{p-1}})^{\frac{q}{p}} \, dw  \Bigg).
\end{equation}
Here  $C>0$ is a constant 
 depending only on  $q$, $p$, $n$, $\omega$,  $\Lambda$,   $M_0$, $s$, and $[w]_{A_s}$.
\end{theorem}
The proof of this key result will be given in Subsection~\ref{sub:Lebesgue-Spaces}, and we will demonstrate
in Subsection~\ref{sec:Morrey-Spaces} that the 
gradient estimates  in weighted Morrey spaces (Theorem~\ref{thm:weighted-morrey}) can 
be derived  as a direct consequence of Theorem~\ref{thm:main}. Notice  that by combining  
 Theorem~\ref{thm:main} with  Muckenhoup's  strong type  weighted estimate for the maximal function
(see \cite{M} and \cite[Theorem~9.1.9]{Gra}), we immediately get: 

\begin{corollary}[weighted $L^q$ space estimate]\label{cor:weighted}
%Let $\K\subset \R$ be an open interval. 
Let $\A$ satisfy
\eqref{structural-reference-1}--\eqref{structural-reference-3} with $p>1$.
Then for any $q>p$, $M_0>0$, and any weight $w\in A_{\frac{q}{p}}$, there exists a  constant  $\delta>0$  such that: if
 $\lambda>0$, $\theta>0$, \eqref{ABCD} holds,
%\begin{equation*}
%\sup_{0<\rho\leq \frac15}\sup_{y\in B_\frac32} \Theta_{B_\rho(y)}(\A)\leq  \delta,
%\end{equation*}
and  $u$ is a weak solution of \eqref{ME} satisfying 
 $\|u\|_{L^\infty(B_2)}\leq \frac{M_0}{\lambda \theta}$,  we have
\begin{equation*}
\fint_{B_1} |\nabla u|^q\, dw \leq C\Bigg(\|\nabla u\|_{L^p(B_2)}^q  +
 \fint_{B_2} |\F|^{\frac{q}{p-1}}\, dw\Bigg).
\end{equation*}
Here  $C,\, \delta$ are constants 
 depending only on  $q$, $p$, $n$, $\omega$, $\Lambda$, $M_0$, and $[w]_{A_{\frac{q}{p}}}$.
\end{corollary}
%It is known that Corollary~\ref{cor:weighted} 
%implies   gradient estimates  in Morrey spaces by a suitable choice of the weight $w$ (see for instance the argument used in \cite[Page~2506]{MP2}). 
%Nevertheless, 
%we will see later that even the  

\section{Weighted Morrey space estimates for the maximal function}\label{sub:two-weight-est}
This section is devoted to deriving a weighted Morrey space estimate for the maximal function. 
Let us first recall standard maximal operators used in this paper.
\begin{definition}\label{def:maximal} Let $f\in L^1_{loc}(\R^n)$. The centered Hardy--Littlewood maximal function of 
$f$ is defined by
\[
 \M (f)(x) = \sup_{\rho>0}\fint_{B_\rho(x)}{|f(y)|\, dy}.
\] 
The uncentered Hardy-Littlewood maximal operator $\tilde\M$ is defined over balls   by
\[
\tilde\M( f)(x) =\sup_{B\ni x}\fint_{B} |f(y)|\, dy.
\]
If  $U\subset \R^n$ is an open set and $f\in L^1(U)$, then we   denote
$ \M_U (f) = \M (f \chi_U )$ and $ \tilde\M_U (f) = \tilde\M ( f \chi_U )$. 
 \end{definition}
 
 The following characterization about the   two-weight norm   inequality
\begin{equation}\label{2-weight-ineq}
  \|\tilde\M(f)\|_{L^p_w(\R^n)}\leq C \|f\|_{L^p_v(\R^n)}
\end{equation}
is due to Sawyer \cite{S} and Kairema \cite{Ka}:
\begin{theorem}[two-weight norm inequality]\label{2-weight-est}
 Let $w,\, v$ be two weights and  $1<q<\infty$. Then we have
 \[
  \|\tilde\M(f)\|_{L^q_w(\R^n)}\leq C_n q' [w, v^{1-q'}]_{S_q} \|f\|_{L^q_v(\R^n)}\quad \forall f\in L^q_v(\R^n),
 \]
 where 
 \[
  [w, \sigma]_{S_q} := \sup_{B}\Bigg(\frac{1}{\sigma(B)}\int_{B}[ \tilde\M(\sigma \chi_B)]^q w\, dx \Bigg)^{\frac1q}. 
 \]
Moreover, if there exists $C>0$ such that \eqref{2-weight-ineq} holds for all  $f\in L^q_v(\R^n)$
then  $[w, v^{1-q'}]_{S_q}<\infty$.
\end{theorem}
\begin{remark}\label{rmk:necessary} From the definition of $\tilde\M$, it is obvious that
 \[
  [w, v^{1-q'}]_{A_q}:= \sup_{B}{ \Big(\fint_{B} w(x)\, dx\Big) \Big(\fint_{B} v(x)^{1-q'}\, dx\Big)^{q-1}  } \leq [w, v^{1-q'}]_{S_q}^q. 
 \]
\end{remark}

Our main result in this subsection is:
\begin{theorem}\label{thm:maximal-est} Let $w,\, v$ be two weights and  $1<q<\infty$.
Assume that $\varphi$ and $\phi$ are two positive functions on the set of nonempty open balls in  $\R^n$.
Then for any bounded open set $U\subset \R^n$, we have
\begin{equation}\label{MI}
 \|\tilde\M_U(f)\|_{\calM^{q,\varphi}_w(U)}\leq C(n,q)  [w, v^{1-q'}]_{S_q} [w\varphi, v\phi]_q \|f\|_{\calM^{q,\phi}_v(U)}
 \end{equation}
 for any  function $f\in L^1(U)$. Here,
 \begin{equation*}\label{wmu1-vmu2}
 [w\varphi, v\phi]_q :=\sup_{y\in U,\, 0<r\leq \diam(U)}  \left[\frac{v(B_{2r}(y))}{w(B_r(y))} \frac{\phi(B_{2r}(y))^{-1}}{\varphi(B_r(y))^{-1}}
    +
     \sup_{2r\leq  s\leq 2\diam(U)}\frac{v(B_s(y))}{w(B_s(y))} \frac{\phi(B_s(y))^{-1}}{\varphi(B_{r}(y))^{-1}} \right]^{\frac1q} .
 \end{equation*}
\end{theorem}
\begin{proof}
 Without loss of generality, we can assume that  $f\in L^1(U)$ and $f\geq 0$.

For each  $y\in U$ and $0<r\leq \diam (U)$, let us write 
 $f = f\chi_{B_{2r}(y)} +  f\chi_{\R^n\setminus B_{2r}(y)}=: f_1 + f_2$. Then $\tilde\M_U(f) \leq\tilde\M_U(f_1) + \tilde\M_U(f_2)$, and hence
 \begin{align}\label{eq:inter-max-est}
  \|\tilde\M_U(f)\|_{L_w^q(B_r(y)\cap U)}
    %\leq \|\tilde\M(\chi_U f_1)\|_{L_w^q(B_r(y)\cap U)} +\|\tilde\M(\chi_U f_2)\|_{L_w^q(B_r(y)\cap U)}
  &\leq \|\tilde\M( f_1 \chi_U)\|_{L_w^q(\R^n)} +\|\tilde\M( f_2 \chi_U)\|_{L_w^q(B_r(y)\cap U)}\nonumber\\
  &\leq C_n q' [w, v^{1-q'}]_{S_q}  \| f_1 \chi_U\|_{L_v^q(\R^n)} +\|\tilde\M( f_2 \chi_U)\|_{L_w^q(B_r(y)\cap U)},
 \end{align}
 where we have used Theorem~\ref{2-weight-est} to obtain the last inequality.
 %and $C_1 =C(n,q)$.
Now for any $x\in B_r(y)\cap U$ and any ball $B_t(z)\ni x$, we clearly have $B_t(z)\subset B_{2t +r}(y)$. Therefore, by considering separately the case $t\leq r/2$ and the case $t>r/2$ we see that 
 \[
  \tilde\M(\chi_U f_2)(x) \leq 4^n \sup_{s>2r}\frac{1}{|B_s(y)|}\int_{B_s(y) \setminus B_{2r}(y)} f \chi_U\, dz=: 4^n I_r(y) \quad \forall x\in B_r(y)\cap U.
 \]
This together with \eqref{eq:inter-max-est} yields
\begin{equation}\label{eq:step1}
  \|\tilde\M_U(f)\|_{L_w^q(B_r(y)\cap U)}
  \leq C_n q' [w, v^{1-q'}]_{S_q} \|f\|_{L_v^q(B_{2 r}(y)\cap U)} +4^n  w(B_r(y))^{\frac1q} I_r(y)
 \end{equation}
 for every $y\in U$ and $0<r\leq\diam (U)$.
  Using definition \eqref{weighted-morrey} and \eqref{eq:step1}, we get
 \begin{align*}
  \|\tilde\M_U(f)\|_{\calM^{q,\varphi}_w(U)}
  &= \sup_{y\in U,\, r\leq \diam(U)}  \Bigg(\frac{\varphi(B_r(y))}{w(B_r(y))}\int_{B_r(y)\cap U} |\tilde\M_U(f)|^q \, dw\Bigg)^{\frac1q}\\
  &\leq C_n q' [w, v^{1-q'}]_{S_q} \sup_{y\in U,\, r\leq \diam(U)} \Bigg( \frac{v(B_{2r}(y))}{w(B_r(y))} 
  \frac{\varphi(B_r(y))}{\phi(B_{2r}(y))} \Bigg)^{\frac1q}
  \Bigg(\frac{\phi(B_{2r}(y))}{v(B_{2r}(y))}\int_{B_{2 r}(y)\cap U} |f|^q \, dv\Bigg)^{\frac1q}\\ 
  &\quad + 4^n  \sup_{y\in U,\, r\leq \diam(U)}  \varphi(B_r(y))^{\frac1q} I_r(y).
 \end{align*}
 When $r\geq \diam(U)/2$, we have $U\subset B_{2 r}(y)$ and hence $I_r(y)=0$. Therefore, we deduce from the above inequality that
  \begin{align}\label{ABC}
  \|\tilde\M_U(f)\|_{\calM^{q,\varphi}_w(U)}
  &\leq C_n q' [w, v^{1-q'}]_{S_q} \|f\|_{\calM^{q,\phi}_v(U)} 
\sup_{y\in U,\, r\leq \diam(U)} 
  \Bigg(\frac{v(B_{2r}(y))}{w(B_r(y))} \frac{\varphi(B_r(y))}{\phi(B_{2r}(y))} \Bigg)^{\frac1q}\nonumber \\
  &\quad + 4^n \sup_{y\in U,\, r<\frac12 \diam(U)} \varphi(B_r(y))^{\frac1q} I_r(y).
 \end{align}
  We use H\"older inequality to estimate $I_r(y)$ for $r<\frac12 \diam(U)$ as follows
 \begin{align*}
  I_r(y)
  &\leq  \sup_{2r<s} \fint_{B_s(y)}  f \chi_U\, dz =\sup_{2r<s\leq \diam(U)} \fint_{B_s(y)}  f \chi_U\, dz\\
  &\leq \sup_{2r<s\leq \diam(U)}\frac{1}{|B_s(y)|} \Big(\int_{B_s(y)} v^{\frac{-1}{q-1}}\, dz\Big)^{\frac{q-1}{q}} \Big( \int_{B_s(y)} |f \chi_U|^q \, dv(z)\Big)^{\frac1q}\nonumber\\
  &=\sup_{2r<s\leq \diam(U)}\left[ \Big(\fint_{B_s(y)} v\, dz\Big) \Big(\fint_{B_s(y)} v^{\frac{-1}{q-1}}\, dz\Big)^{q-1}\right]^{\frac1q} 
  \Big( \fint_{B_s(y)}  | f \chi_U|^q \, dv(z)\Big)^{\frac1q}.
 \end{align*}
 It then follows from Remark~\ref{rmk:necessary} that
 \begin{align*}
  I_r(y)
  \leq [w, v^{1-q'}]_{S_q} \sup_{2r<s\leq \diam(U)} \Big(\frac{v(B_s(y))}{w(B_s(y))} \Big)^{\frac1q}\Big( \fint_{B_s(y)} |f \chi_U|^q \, dv(z)\Big)^{\frac1q}.
 \end{align*}
 Plugging this  into \eqref{ABC}, we conclude that
\begin{align*}
  &\|\tilde\M_U(f)\|_{\calM^{q,\varphi}_w(U)}
    \leq C_n q' [w, v^{1-q'}]_{S_q} \|f\|_{\calM^{q,\phi}_v(U)} \sup_{y\in U,\, r\leq \diam(U)}  \Bigg( 
  \frac{v(B_{2r}(y))}{w(B_r(y))} \frac{\varphi(B_r(y))}{\phi(B_{2r}(y))} \Bigg)^{\frac1q} \\
    &\quad + 4^n [w, v^{1-q'}]_{S_q}
   \sup_{y\in U,\, r< \frac12\diam(U)}  \sup_{2r<s\leq \diam(U)} \Bigg( \frac{v(B_s(y))}{w(B_s(y))} \frac{\varphi(B_r(y))}{\phi(B_s(y))} \Bigg)^{\frac1q}\Bigg(
    \frac{\phi(B_s(y))}{v(B_s(y))}\int_{B_s(y)\cap U} |f(z)|^q \, dv(z)\Bigg)^{\frac1q}\\
    &\quad \leq C(n,q)[w, v^{1-q'}]_{S_q} [w\varphi, v\phi]_q \|f\|_{\calM^{q,\phi}_v(U)}.
    %\sup_{y,\, r>0} \left[\frac{v(B_{2r}(y))}{w(B_r(y))} \frac{\mu_2(B_{2r}(y))}{\mu_1(B_r(y))}
    %+
     %\sup_{s>2r}\frac{v(B_s(y))}{w(B_s(y))} \frac{\mu_2(B_s(y))}{\mu_1(B_r(y))} \right]^{\frac1q} 
 %\\
  %  &\leq  C_* \|f\|_{\calL^{q,\mu_2}_v(\R^n)}.
 \end{align*}
 This gives estimate \eqref{MI} as desired.
\end{proof}
\begin{remark}
By inspecting the proof we see that inequality \eqref{MI} also holds true if $U$ is replaced by $\R^n$. 
\end{remark}
Given Theorem~\ref{thm:maximal-est}, it is desirable to have concrete conditions on $w$, $v$, $\varphi$, and $\phi$ ensuring that
$[w, v^{1-q'}]_{S_q}<\infty$ and $ [w\varphi, v\phi]_q<\infty$. The finiteness of the constant $[w, v^{1-q'}]_{S_q}$
has been investigated extensively (see \cite{F,Ne,P,PR,W}) and for this to be true it is necessary that
$[w, v^{1-q'}]_{A_q}<\infty$ (see Remark~\ref{rmk:necessary}). In particular, we have from \cite[Corollary~1.3]{P} and \cite[Corollary~1.4]{PR} that
$[w, v^{1-q'}]_{S_q}<\infty$ if any of the 
following two conditions is satisfied:
\begin{enumerate}
 \item[({\bf A})] There exists $r>1$ such that
 \[
   \sup_{B}{ \Big(\fint_{B} w\, dx\Big) \Big(\fint_{B} v^{r(1-q')}\, dx\Big)^{\frac{q-1}{r}}  }<\infty.
 \]
  \item[({\bf B})] $[w, v^{1-q'}]_{A_q}<\infty$
and $v^{1-q'}\in A_\infty$. 
\end{enumerate}
We refer readers to \cite{P,PR} for more general conditions guaranteeing the finiteness of $[w, v^{1-q'}]_{S_q}$. 
On the other hand, it is clear that $[w\varphi, v\phi]_q<\infty$
if $w$ is doubling and there exists a constant $C_*>0$ such that
\begin{equation}\label{wmu1-vmu2-sufficient}
\sup_{2r \leq s\leq 2\diam(U)}  \frac{v(B_s(y))}{w(B_s(y))} \frac{1}{\phi(B_s(y))} \leq C_* \, \frac{1}{\varphi(B_r(y))} \quad \mbox{for all $y\in U$
and $0<r\leq \diam(U)$}. 
 \end{equation}
%Notice that condition \eqref{wmu1-vmu2-sufficient} is trivially satisfied if $v(B_s(y))\mu_2(B_s(y)) \leq C w(B_r(y)) \mu_1(B_r(y))$ for all
% $s\geq 2r$ and $y\in \R^n$.
For convenience, we summarize the above discussions into two separate results:
\begin{corollary}[two-weight case] \label{cor:two-weight-case} Let $U\subset \R^n$ be a bounded open set. Assume that the hypotheses of Theorem~\ref{thm:maximal-est} and 
condition \eqref{wmu1-vmu2-sufficient}  hold, and $w$ is doubling. 
Assume in addition that either ({\bf A}) or ({\bf B}) is satisfied.
Then $[w, v^{1-q'}]_{S_q}<\infty$ and there exists a constant $C>0$ depending only on $n$, $q$, $C_*$, the doubling constant for $w$, and 
$[w, v^{1-q'}]_{S_q}$  such that
\begin{equation*}
 \|\tilde\M_U(f)\|_{\calM^{q,\varphi}_w(U)}\leq C \|f\|_{\calM^{q,\phi}_v(U)}
 \qquad \forall f\in L^1(U).
 \end{equation*}
 %for any function $f\in L^1_{loc}(\R^n)$. 
\end{corollary}

\begin{corollary}[one-weight case] \label{cor:one-weight-case}  Let $1<q<\infty$, $w\in A_q$, and $U\subset \R^n$ be a bounded open set.
Assume that $\varphi$ and $\phi$ are two positive functions on the set of  open balls in  $\R^n$
such that there exists  $C_*>0$ satisfying
\begin{equation}\label{mu1-mu2-sufficient}
\sup_{2r \leq s\leq 2\diam(U)}  
\phi(B_s(y))^{-1} \leq C_* \, \varphi(B_r(y))^{-1} \quad \mbox{for all $y\in U$
and $0<r\leq  \diam(U)$}. 
 \end{equation}
Then there exists  a constant $C>0$ depending only on $n$, $q$, $C_*$, and $[w]_{A_q}$ such that
\begin{equation*}
 \|\tilde\M_U(f)\|_{\calM^{q,\varphi}_w(U)}\leq C \|f\|_{\calM^{q,\phi}_w(U)}\quad\mbox{for any } f\in L^1(U).
 \end{equation*}
% where $C>0$ depending only on $n$, $q$, $C_*$, and $[w]_{A_q}$.
\end{corollary}
\begin{proof}
 Since $w\in A_q$, we have from Lemma~\ref{weight:basic-pro}  that $w^{1-q'}\in A_{q'}\subset A_\infty$. Moreover, $[w, w^{1-q'}]_{A_q}= [w]_{A_q}$.
 Therefore, the conclusion is a consequence of Corollary~\ref{cor:two-weight-case}  and the first inequality in
 \eqref{strong-doubling}.
\end{proof}
Notice that if  $\phi=\varphi$, then  condition \eqref{mu1-mu2-sufficient}  trivially holds if
$\varphi\in \mathcal B_0$ defined in Definition~\ref{class-B}.
%the function $s\in (0,\infty) \mapsto \varphi(B_s(y))$ is nondecreasing  for all
%$y\in \R^n$. 
In particular, we recover the estimate by Chiarenza and Frasca in \cite{ChF} by taking 
$\varphi(B) := |B|^{\frac{\lambda}{n}}$ with $0\leq \lambda\leq n$. Furthermore,
Corollary~\ref{cor:one-weight-case} extends the weighted Morrey space estimates obtained recently in 
\cite{GMOS,GORS,KS,Na}.

\section{Approximating gradients of solutions }\label{approximation-gradient}
The purpose of this section is to 
show that gradients of weak solutions to \eqref{ME} can be approximated by 
bounded gradients in an invariant way. Proposition~\ref{lm:localized-compare-gradient} below is an extension of
\cite[Lemma~5.3]{NP} and \cite[Lemma~4.2]{BPS} where the modulus $\omega$ is respectively required to be 
$\omega(s)=s$ (Lipschitz continuity) and $\omega(s)=s^\alpha$ with $\alpha\in (0,1)$ (H\"older 
continuity). Here by a simple modification  of the arguments in \cite{BPS} 
we prove that the approximation holds true for any  modulus of continuity $\omega$.
This   result plays a crucial role in our derivation of 
gradient  estimates   for  solutions of   \eqref{GE-10}. 
\begin{proposition}
%{(Lemma~5.3 in \cite{NP})}
\label{lm:localized-compare-gradient}  Let $\A$ satisfy \eqref{structural-reference-1}--\eqref{structural-reference-3} with $p>1$, and let $M_0>0$.
For any $\e\in (0,1]$, there exist small positive constants $\delta$ and $\sigma$  depending only on $\e$,   $p$,   
$n$, $\omega$, $\Lambda$,
 and $M_0$ such that: 
if $ \lambda>0$, $\theta>0$, $r>0$,
\begin{equation*}
\Theta_{B_{3 \sigma r}}(\A) \leq \delta \quad \mbox{and}\quad \fint_{B_{4r}}  |\F|^{p'} \, dx  \leq \delta,
\end{equation*}
and $u\in W_{loc}^{1,p}(B_{4 r})$ is a weak solution of $\div \Big[\frac{\A(x,\lambda\theta u, \lambda \nabla u)}{\lambda^{p-1}}\Big] = \div \F\,$ in $B_{4 r}$ satisfying
\begin{equation*}
 \|u\|_{L^\infty(B_{4r})}\leq \frac{M_0}{\lambda\theta},\quad  \fint_{B_{4 r}}{|\nabla  u|^p \, dx}\leq 1,
 \quad\mbox{and}\quad
 \fint_{B_{4\sigma r}}{|\nabla  u|^p dx}\leq 1,
\end{equation*}
then 
\begin{equation*}
\fint_{B_{3\sigma r}}{|\nabla  u - \nabla  v|^p\, dx}\leq \e^p
\end{equation*}
for some function $v\in  W^{1,p}(B_{3 \sigma r})$  with
\begin{equation*}
\|\nabla  v\|_{L^\infty(B_{2 \sigma  r})}^p
\leq C(p,n, \Lambda) \fint_{B_{3 \sigma r}}{|\nabla  v|^p \, dx}.
\end{equation*}
\end{proposition}

\begin{remark}\label{rm:translation-invariant}
Since the class of our equations is invariant under the transformation $x\mapsto x+y$, Proposition~\ref{lm:localized-compare-gradient} still holds true if $B_r$ is replaced by $B_r(y)$.
\end{remark}

We will use the fact (see  the proof of  \cite[Lemma~1]{T}) that condition \eqref{structural-reference-1}
  implies that  
  \begin{align}\label{structural-consequence}
&\big\langle  \A(x,z,\xi) -\A(x,z,\eta), \xi-\eta\big\rangle \geq 
\left \{
\begin{array}{lcll}
  4^{1-p}\Lambda^{-1} |\xi-\eta|^p &\text{if}\quad p\geq 2,\\
4^{-1}\Lambda^{-1} \big(|\xi| +|\eta|)^{p-2} |\xi-\eta|^2 &\quad\,\,\text{ if}\quad 1 <p<2
\end{array}\right. 
\end{align}
for a.e. $x\in B_{10}$, all $z\in \overline\K$, and all $\xi,\, \eta\in \R^n$. The next result is the first step of proving Proposition~\ref{lm:localized-compare-gradient}. 
\begin{lemma}\label{lm:step1}
For any $\e>0$, there exist small positive constants  $\delta$ and $\sigma$ depending only on $\e$,  $p$,  
$n$, $\omega$, $\Lambda$,   and $M_0$ such that: 
if $ \lambda>0$, $\theta>0$, 
\begin{equation*}
  \fint_{B_{4}}  |\F|^{p'} \, dx  \leq \delta,
\end{equation*}
and $u\in W_{loc}^{1,p}(B_{4})$ is a weak solution of $\div \Big[\frac{\A(x,\lambda\theta u, \lambda \nabla u)}{\lambda^{p-1}}\Big] = \div \F\,$ in $B_{4}$ satisfying
\begin{equation*}
 \|u\|_{L^\infty(B_{4})}\leq \frac{M_0}{\lambda\theta} \quad\mbox{ and }\quad \fint_{B_{4}}{|\nabla  u|^p \, dx}\leq 1,
\end{equation*}
then 
\begin{equation*}
\fint_{B_{4\sigma }}{|\nabla  u - \nabla  f|^p\, dx}\leq \e^p,
\end{equation*}
where $f$ is a weak solution of 
\begin{equation} \label{eq-f} 
 \left\{
 \begin{alignedat}{2}
  &\div \Big[\frac{\A(x,\lambda\theta \bar h_{B_{4 \sigma}}, \lambda \nabla f)}{\lambda^{p-1}}\Big]=0\quad \mbox{in}\quad B_{ 4 \sigma},\\\
&f=h \qquad\qquad\quad \qquad\qquad\,\,\, \mbox{on}\quad \partial B_{ 4 \sigma}
 \end{alignedat} 
  \right.
\end{equation}
with $\bar h_{B_{4 \sigma}} := \fint_{B_{4 \sigma}} h(x)\, dx$ and $h$ being a weak solution of 
\begin{equation} \label{eq-h}
 \left\{
 \begin{alignedat}{2}
  &\div \Big[\frac{\A(x,\lambda\theta u, \lambda \nabla h)}{\lambda^{p-1}}\Big]=0\quad \mbox{in}\quad B_{4},\\\
&h=u \qquad\qquad\quad \qquad\quad \,\,\,\,\,\mbox{on}\quad \partial B_{4}.
 \end{alignedat} 
  \right.
\end{equation}
\end{lemma}
The proof of this result relies on  the following classical regularity result for 
weak solutions of homogeneous equation \eqref{eq-h}
  (see for instance \cite[estimates~(6.~64) and (7.45)]{Gi}).
\begin{lemma}[classical interior regularity]\label{lm:inter-regularity} Assume that $\ba(x,\xi)$ satisfies 
\eqref{structural-reference-1}--\eqref{structural-reference-2} with $p>1$. 
Let  $R\in (0,10)$ and $w\in W^{1,p}(B_R)$ be a weak solution of $\div \ba(x,\nabla w)=0$ in $B_R$. Then there exist constants $p_0\in (p,\infty)$,  $\beta \in (0,1)$, and $C>0$ depending only on $p$, $n$, and $\Lambda$ such that  for every $r\in (0,R)$ we have 
\begin{align*}
&\Big(\fint_{B_r} |\nabla w|^{p_0} dx\Big)^{\frac1p_0} \leq  C \Big(\fint_{B_R} |\nabla w|^{p} dx\Big)^{\frac1p}\quad\,\,\mbox{ and }\, \, \quad \sup_{x\in B_r}{|w(x) -\bar w_{B_r}|}  
\leq C \big(\frac{r}{R}\big)^\beta
\, \|w\|_{L^\infty(B_R)}.
\end{align*} 
\end{lemma}

\begin{proof}[\bf Proof of Lemma~\ref{lm:step1}] 
We present the complete proof only for the case $p\geq 2$ by using an idea in \cite{BPS}, and will indicate at the end how the argument can be modified for the case $1<p<2$.
Consider $p\geq 2$ and, for convenience, let $\tilde \A(x, z,\xi) := \frac{\A(x,\lambda \theta z, \lambda \xi)}{\lambda^{p-1}}$.  We write
\[
\nabla u -\nabla f = \nabla (u-h) +\nabla (h-f)
\]
and will estimate $\| \nabla (u-h)\|_{L^p(B_{4\sigma})}$ and  $\| \nabla (h-f)\|_{L^p(B_{4\sigma})}$.
   By using $u-h$ as a test function in the equations for $u$ and $h$ we have
\begin{equation*}
\int_{B_4}\langle \tilde\A(x,u,\nabla u)- \tilde \A(x,u,\nabla h), \nabla(u-h) \rangle \, dx = \int_{B_4}\langle \F, \nabla(u-h) \rangle \, dx.
\end{equation*}
We then  use \eqref{structural-consequence} to bound the above left hand side from below. As a consequence, we obtain
\begin{equation*}
\int_{B_4}| \nabla(u-h) |^p dx\leq 4^{p-1}\Lambda \int_{B_4}\langle \F, \nabla(u-h) \rangle \, dx.
\end{equation*}
Then we infer from Young's inequality that
\begin{equation}\label{u-h-0}
 \int_{B_4}| \nabla(u-h) |^p dx\leq C \int_{B_4}| \F|^{p'} dx
\end{equation}
yielding
\begin{equation}\label{u-h}
\fint_{B_{4 \sigma}}| \nabla(u-h) |^p dx\leq \frac{C }{\sigma^n} \fint_{B_4}| \F|^{p'} dx.
\end{equation}
  By using $h-f$ as a test function in the equations for $h$ and $f$ we have
\begin{equation*}
\int_{B_{4 \sigma}}\langle \tilde \A(x,\bar h_{B_{4 \sigma}},\nabla f), \nabla(h-f) \rangle \, dx = \int_{B_{4 \sigma}}\langle \tilde \A(x,u,\nabla h), \nabla(h-f) \rangle \, dx.
\end{equation*}
This together with  \eqref{structural-consequence} gives
\begin{align}\label{est-nabla-h-f}
\int_{B_{4 \sigma}} |\nabla (h-f)|^p\, dx 
&\leq 4^{p-1}\Lambda \int_{B_{4 \sigma}}\langle \tilde\A(x,\bar h_{B_{4 \sigma}},\nabla h)- \tilde \A(x,\bar h_{B_{4 \sigma} },\nabla f), \nabla(h-f) \rangle \, dx\nonumber\\
&= 4^{p-1}\Lambda \int_{B_{4 \sigma}}\langle \tilde\A(x,\bar h_{B_{4 \sigma}},\nabla h)- \tilde \A(x,u,\nabla h), \nabla(h-f) \rangle \, dx\nonumber\\
&\leq  4^{p-1}\Lambda \int_{B_{4 \sigma}}\min{\{2\Lambda, \omega(\lambda \theta |u-\bar h_{B_{4 \sigma}}|)\}}\,\, |\nabla h|^{p-1} |\nabla(h-f)| \, dx.
\end{align}
As a consequence of \eqref{est-nabla-h-f} and Young's inequality, we obtain
\[
\int_{B_{4\sigma}} |\nabla (h-f)|^p\, dx \leq C(p,\Lambda) \int_{B_{4 \sigma}} |\nabla h|^p \, dx.
\]
On the other hand, it follows  from   equation \eqref{eq-h} for $h$ that
$\int_{B_{4 \sigma}} |\nabla h|^p \, dx \leq C  \int_{B_{4}} |h|^p\, dx$.
Thus we conclude that
\begin{equation*}
\fint_{B_{4\sigma}} |\nabla (h-f)|^p\, dx \leq  \frac{C}{\sigma^n} \int_{B_{4}} |h|^p\, dx\leq \frac{ C_*}{\sigma^n}   \|h\|_{L^\infty(B_4)}^p\leq \frac{C_*}{\sigma^n}    \|u\|_{L^\infty(B_4)}^p\leq \frac{C_*}{\sigma^n} \Big(\frac{M_0}{\lambda \theta}\Big)^p.
\end{equation*}
This  together with \eqref{u-h} gives  the desired conclusion if $ \frac{C_*}{\sigma^n}  \Big(\frac{M_0}{\lambda \theta }\Big)^p 
\leq \e^p/2$. We hence only need to consider the case \begin{equation}\label{small-para} 
\frac{C_*}{\sigma^n}    \Big(\frac{M_0}{\lambda \theta }\Big)^p >\frac{\e^p}{2}.
\end{equation}
 For this, note first that \eqref{u-h-0} and the assumption yield
\begin{equation}\label{control-Lp-Dh}
\|\nabla h\|_{L^p(B_4)} \leq \|\nabla (h-u)\|_{L^p(B_4)} +\|\nabla u\|_{L^p(B_4)}\leq C\left[\Big(\int_{B_4} |\F|^{p'}dx\Big)^{\frac1p}  +1 \right]\leq C.
\end{equation}
We deduce from
 \eqref{est-nabla-h-f}, Young's inequality,  the higher integrability for $\nabla h$ given by Lemma~\ref{lm:inter-regularity}, and \eqref{control-Lp-Dh} that
\begin{align*}
\fint_{B_{4\sigma}} |\nabla (h-f)|^p\, dx 
&\leq C \fint_{B_{4\sigma}} \omega(\lambda \theta |u-\bar h_{B_{4\sigma}}|)^{p'}\, |\nabla h|^p  dx\\
&\leq C
\Big[\fint_{B_{4 \sigma}} \omega(\lambda \theta |u-\bar h_{B_{4 \sigma}}|)^{\frac{p' p_0}{p_0 - p}}  dx\Big]^{\frac{p_0 - p}{p_0}} \Big[\fint_{B_{4\sigma}}  |\nabla h|^{p_0} dx\Big]^{\frac{p}{p_0}}\\
&\leq 
C \Big[\fint_{B_{4 \sigma}} \omega(\lambda \theta |u-\bar h_{B_{4 \sigma}}|)^{\frac{p' p_0}{p_0 - p}}  dx\Big]^{\frac{p_0 - p}{p_0}} \fint_{B_{4}}  |\nabla h|^{p}  dx\leq  C \Big[\fint_{B_{4 \sigma}} \omega(\lambda \theta |u-\bar h_{B_{4\sigma}}|)^{\frac{p' p_0}{p_0 - p}}  dx\Big]^{\frac{p_0 - p}{p_0}}
\end{align*}
with $p_0>p$ and $C>0$ depending only on $p$, $n$, and $\Lambda$. 
As for any $\gamma>0$, we have 
\begin{align*}
\int_{B_{4 \sigma}} \omega(\lambda \theta |u-\bar h_{B_{4 \sigma}}|)^{\frac{p' p_0}{p_0 - p}}  dx
&= \int_{\{B_{4\sigma}: \lambda \theta |u-\bar h_{B_{4\sigma}}| \leq \gamma  \}}
\omega(\lambda \theta |u-\bar h_{B_{4 \sigma}}|)^{\frac{p' p_0}{p_0 - p}}  dx\\
&\qquad \qquad \qquad \qquad \qquad +\int_{\{B_{4 \sigma}: \lambda \theta |u-\bar h_{B_{4\sigma}}| > \gamma  \}} \omega(\lambda \theta |u-\bar h_{B_{4 \sigma}}|)^{\frac{p' p_0}{p_0 - p}}  dx \\
&\leq |B_{4\sigma}| \, \omega(\gamma)^{\frac{p' p_0}{p_0 - p}}  +\frac{\omega(2 M_0)^{\frac{p' p_0}{p_0 - p}} }{\gamma^p}  \int_{B_{4 \sigma}} \big(\lambda \theta |u-\bar h_{B_{4 \sigma}}|\big)^p dx,
\end{align*}
we infer that
\begin{equation}\label{est-delta}
\fint_{B_{4 \sigma}} |\nabla (h-f)|^p\, dx 
\leq C  \omega(\gamma)^{p'} + C \frac{\omega(2 M_0)^{p'}}{\gamma^{\frac{p(p_0 -  p)}{p_0}}}  \Big[(\lambda \theta)^p\fint_{B_{4 \sigma}} |u-\bar h_{B_{4 \sigma}}|^p dx\Big]^{\frac{p_0 - p}{p_0}}\quad \forall \gamma>0.
\end{equation}
Let us now estimate the  last integral in \eqref{est-delta}. Using Sobolev's inequality and the oscillation estimate in Lemma~\ref{lm:inter-regularity}, we have
\begin{align*}
\fint_{B_{4\sigma} }|u-\bar h_{B_{4 \sigma}}|^p\, dx 
&\leq 2^{p-1} \Big[\frac{1}{|B_{4\sigma}| } \int_{B_{4}} |u- h|^p\, dx   + \fint_{B_{4 \sigma}} |h-\bar h_{B_{ 4\sigma}}|^p\, dx \Big]\\
& \leq C  \Big[ \sigma^{-n}\int_{B_{4}} |\nabla (u- h)|^p\, dx   +\big(\sigma^{\beta }  \| h\|_{L^\infty(B_4)}\big)^p \Big]. 
\end{align*}
This together with \eqref{small-para}  and the fact $ \| h\|_{L^\infty(B_4)}\leq M_0/\lambda\theta$ yields
\begin{align*}
(\lambda \theta)^p\fint_{B_{4 \sigma} }|u-\bar h_{B_{4\sigma}}|^p\, dx 
\leq C  \Big[\sigma^{-2 n}\big(\frac{M_0 }{\e }\big)^p  \int_{B_{4}} |\nabla (u- h)|^p\, dx   +\big(M_0 \sigma^{\beta } \big)^p \Big]. 
\end{align*}
Plugging this estimate into \eqref{est-delta} gives 
\begin{align*}
\fint_{B_{4\sigma}} |\nabla (h-f)|^p\, dx 
\leq  C  \omega(\gamma)^{p'}  +  C  M_0^{\frac{p(p_0 -  p)}{p_0}} \omega(2 M_0)^{p'} \left[\frac{\sigma^{-2n} }{(\gamma \e  )^p} 
\int_{B_{4}} |\nabla (u- h)|^p\, dx   +
\big(\frac{\sigma^{\beta }}{\gamma }\big)^p
 \right]^{\frac{p_0 -  p}{p_0}}.
\end{align*}
By combining this with \eqref{u-h} and using \eqref{u-h-0} we obtain 
\begin{align}\label{pertur-p>2}
\fint_{B_{4\sigma}} |\nabla (u-f)|^p\, dx
&\leq \frac{C}{\sigma^n}\fint_{B_4}|\F|^{p'} dx\nonumber\\
&\quad + C \omega(\gamma)^{p'} + C M_0^{\frac{p(p_0 -  p)}{p_0}} \omega(2 M_0)^{p'} 
\left[\frac{\sigma^{-2 n} }{(\gamma \e  )^p } \fint_{B_4}|\F|^{p'}dx +  \big(\frac{\sigma^{\beta }}{\gamma }\big)^p \right]^{\frac{p_0 -  p}{p_0}}
\end{align}
for every $\gamma>0$. From this, we get the desired conclusion by  choosing $\gamma$ small first,
then $\sigma$, and  $\delta$ last.
 For the case $1<p<2$, the proof is similar with some 
slight adjustments as follows, which can also be found in the proofs of
Lemma~4.3 and Lemma~4.6 in \cite{BPS}. By arguing as above but using 
 \eqref{structural-consequence} together with   \cite[Lemma~3.1]{NP}, then in place  of \eqref{u-h} and \eqref{est-nabla-h-f}  we now 
 have  for every $\tau_1, \,\tau_2>0$ small that
\begin{equation*}
\fint_{B_{4 \sigma}}| \nabla(u-h) |^p dx\leq \frac{C}{\sigma^n} \Bigg[
\tau_1 \fint_{B_4} |\nabla u|^p\, dx +   \tau_1^{\frac{p-2}{p-1}} \fint_{B_4}| \F|^{p'} dx\Bigg]
\leq \frac{C}{\sigma^n} \Bigg[
\tau_1  +   \tau_1^{\frac{p-2}{p-1}} \fint_{B_4}| \F|^{p'} dx\Bigg]
\end{equation*}
and 
\begin{align*}
\int_{B_{4 \sigma}} |\nabla (h-f)|^p\, dx 
&\leq \tau_2 \int_{B_{4 \sigma}} |\nabla h|^p\, dx +
C_p\Lambda \tau_2^{1-\frac2p} \int_{B_{4 \sigma}}\langle \tilde\A(x,\bar h_{B_{4 \sigma}},\nabla h)- \tilde \A(x,\bar h_{B_{4 \sigma} },\nabla f), \nabla(h-f) \rangle \, dx\nonumber\\
&\leq  \tau_2 \int_{B_{4 \sigma}} |\nabla h|^p\, dx +
C_p\Lambda \tau_2^{1-\frac2p}  \int_{B_{4 \sigma}}\min{\{2\Lambda, \omega(\lambda \theta |u-\bar h_{B_{4 \sigma}}|)\}}\,\, |\nabla h|^{p-1} 
|\nabla(h-f)| \, dx.
\end{align*}
With these changes and by repeating the above lines of argument, we obtain the following version of
\eqref{pertur-p>2} for the case $1<p<2$:
\begin{align*}
\fint_{B_{4\sigma}} |\nabla (u-f)|^p\, dx
&\leq \frac{C}{\sigma^n} \Big[
\tau_1  +   \tau_1^{\frac{p-2}{p-1}} \fint_{B_4}| \F|^{p'} dx\Big]\\
&\quad +C\tau_2 + C \tau_2^{\frac{p-2}{p-1}} \Bigg\{ \omega(\gamma)^{p'} +  M_0^{\frac{p(p_0 -  p)}{p_0}} \omega(2 M_0)^{p'} 
\Big[\frac{\sigma^{-2 n} }{(\gamma \e  )^p } \fint_{B_4}|\F|^{p'}dx +  \big(\frac{\sigma^{\beta }}{\gamma }\big)^p \Big]^{\frac{p_0 -  p}{p_0}}
\Bigg\}.
\end{align*}
We then get the conclusion by  choosing $\tau_2$ small first,  $\gamma$ second, $\sigma$ third,
then $\tau_1$,  and  $\delta$ last.
\end{proof}

To obtain Proposition~\ref{lm:localized-compare-gradient}, our second step is to show that the gradient of the solution $f$ to \eqref{eq-f} 
can be approximated by a bounded gradient. Precisely, we have:
\begin{lemma}\label{lm:step2} 
Let $\e\in (0,1]$, and let  $\sigma$ be its corresponding constant given by  Lemma~\ref{lm:step1}.  
Let  $\F$,  $u$, and $h$  be  as in Lemma~\ref{lm:step1}, and  
assume in addition that $\fint_{B_{4\sigma}} |\nabla u|^p\, dx\leq 1$. 
Suppose that  $f$ is a weak solution of $\div \Big[\frac{\A(x,\lambda\theta \bar h_{B_{4 \sigma}}, \lambda \nabla f)}{\lambda^{p-1}}\Big] = 0$ in $B_{4\sigma}$ and 
$v$ is a weak solution of 
\begin{equation} \label{eq-v}
 \left\{
 \begin{alignedat}{2}
  &\div \Big[\frac{\bar\A_{B_{3\sigma}}(\lambda\theta \bar h_{B_{4 \sigma}}, \lambda \nabla v)}{\lambda^{p-1}}\Big]=0\quad \mbox{in}\quad B_{3\sigma},\\\
&v=f \qquad\qquad \qquad\qquad\qquad\mbox{on}\quad \partial B_{3\sigma}.
 \end{alignedat} 
  \right.
\end{equation}
There exist constants $p_0\in (p,\infty)$ and $C>0$ depending only on $p$, $n$, and $\Lambda$ such that: 
 if $p\geq 2$, then
\begin{equation}\label{est-oscillation}
\fint_{B_{3 \sigma}}{|\nabla  f- \nabla  v|^p\, dx}\leq C \,\,  \Theta_{B_{3 \sigma}}(\A)^{\frac{p_0 - p}{p_0}},
\end{equation}
and if  $1<p<2$, then
\begin{align*}
\fint_{B_{3\sigma}} |\nabla f-\nabla v|^p\, dx 
&\leq 
  \tau 
+ C   \tau^{(1-\frac2p)p'}\Theta_{B_{3\sigma}}(\A)^{\frac{p_0 - p}{p_0}}\quad \mbox{for every}\quad \tau\in (0,1).
\end{align*}
%Here $C>0$ and $p_0>p$ are %constants depending only on $p$, $n%$, and $\Lambda$.
\end{lemma}
\begin{proof}  For convenience, define 
 $\ba(x, \xi) := \frac{\A(x,\lambda \theta \bar h_{B_{4\sigma}}, \lambda \xi)}{\lambda^{p-1}}$ and $\bar\ba_{B_{3\sigma}}(\xi) := \fint_{B_{3\sigma}} \ba(x,\xi) \, dx$.
We first consider the case $p\geq 2$. Then    using \eqref{structural-consequence} we get 
\begin{equation}\label{p-geq-2}
\int_{B_{3\sigma}} |\nabla ( f- v)|^p\, dx 
\leq \Lambda \int_{B_{3\sigma}}\langle \bar\ba_{B_{3\sigma}}(\nabla f)- \bar\ba_{B_{3\sigma}}(\nabla v), \nabla (f- v) \rangle \, dx :=I.
\end{equation}
To estimate the term $I$, we use  $f-v$ as a test function in equation  \eqref{eq-v}  and the equation for $f$ to obtain
\begin{equation*}
\int_{B_{3\sigma}}\langle \bar\ba_{B_{3\sigma}}(\nabla v), \nabla(f-v) \rangle \, dx= \int_{B_{3\sigma}}\langle \ba(x,\nabla f), \nabla(f-v) \rangle \, dx.
\end{equation*}
Therefore,
\begin{align}\label{est-I}
%\int_{B_{3\sigma}} |\nabla (f-v)|^p\, dx 
%&\leq \Lambda \int_{B_{3\sigma}}
%\langle\ba_{B_{3\sigma}}(\nabla f)- 
%\ba_{B_{3\sigma}}(\nabla v), \nabla(f-v) %\rangle \, dx\nonumber\\
I= \Lambda \int_{B_{3\sigma}}\langle \bar\ba_{B_{3\sigma}}(\nabla f)-  \ba(x,\nabla f), \nabla(f-v) \rangle \, dx\leq  \Lambda \int_{B_{3\sigma}} \sup_{\xi\neq 0} \frac{|\ba(x,\xi) -
\bar\ba_{B_{3\sigma}}(\xi)| }{|\xi|^{p-1}}\,\, |\nabla f|^{p-1} |\nabla(f-v)| \, dx.
\end{align}
It follows from \eqref{p-geq-2}--\eqref{est-I},  Young's inequality, the interior higher integrability for $\nabla f$  given 
by Lemma~\ref{lm:inter-regularity}, and the fact $|\ba(x,\xi)|\leq \Lambda |\xi|^{p-1}$ that
\begin{align*}
\fint_{B_{3\sigma}} |\nabla (f-v)|^p\, dx 
&\leq C \fint_{B_{3\sigma}}\Big[ \sup_{\xi\neq 0} \frac{|\ba(x,\xi) -\bar\ba_{B_{3\sigma}}(\xi)| }{|\xi|^{p-1}}\Big]^{p'}\, |\nabla f|^p  dx\\
&\leq C
\left(\fint_{B_{3\sigma}}  \Big[ \sup_{\xi\neq 0} \frac{|\ba(x,\xi) -\bar\ba_{B_{3\sigma}}(\xi)| }{|\xi|^{p-1}}\Big]^{\frac{p' p_0}{p_0 - p}}  dx\right)^{\frac{p_0 - p}{p_0}} \left(\fint_{B_{3\sigma}}  |\nabla f|^{p_0} dx\right)^{\frac{p}{p_0}}\\
&\leq 
C \left(\fint_{B_{3\sigma}}  \sup_{\xi\neq 0} \frac{|\ba(x,\xi) -\bar\ba_{B_{3\sigma}}(\xi)| }{|\xi|^{p-1}} dx\right)^{\frac{p_0 - p}{p_0}} \fint_{B_{4\sigma}}  |\nabla f|^{p}  dx,
\end{align*}
where $C>0$ depends only on $p$, $n$, and $\Lambda$. But  we have from the definition of $\ba$ that
\begin{align*}
\sup_{\xi\neq 0} \frac{|\ba(x,\xi) -\bar\ba_{B_{3\sigma}}(\xi)| }{|\xi|^{p-1}} 
&= \sup_{\eta\neq 0} \frac{|\A(x,\lambda\theta \bar h_{B_{4\sigma}}, \eta) -\bar\A_{B_{3\sigma}}(\lambda\theta \bar h_{B_{4\sigma}}, 
\eta)| }{|\eta|^{p-1}}
%=\sup_{ \eta\neq 0}\frac{|\A(x,z,\eta) - \bar\A_{B_{3\sigma}}(z,\eta)|}{|\eta|^{p-1}}
 \end{align*}
and  Lemma~\ref{lm:step1} yields
\begin{equation}\label{est-nabla-f}
\fint_{B_{4\sigma}}  |\nabla f|^{p}  dx \leq 2^{p-1}\left[ \fint_{B_{4\sigma}} |\nabla f-\nabla u|^p\, dx  +
\fint_{B_{4\sigma}} |\nabla u|^p\, dx \right]\leq 2^{p-1}[\e^p +1 ]\leq 2^p.
\end{equation}
Therefore we conclude that
\begin{align*}
\fint_{B_{3\sigma}} |\nabla (f-v)|^p\, dx 
\leq 
C \left(\fint_{B_{3\sigma}}  \sup_{\xi\neq 0} \frac{|\A(x,\lambda \theta \bar h_{B_{4\sigma}}, \xi) -\bar\A_{B_{3\sigma}}(\lambda 
\theta \bar h_{B_{4\sigma}}, \xi)| }{|\xi|^{p-1}} dx\right)^{\frac{p_0 - p}{p_0}}.
\end{align*}
This together with 
the fact $\lambda\theta \bar h_{B_{4\sigma}}\in \overline{\K} \cap [-M_0, M_0]$ and the definition of $\Theta_{B_{3\sigma}}(\A)$ given by 
\eqref{def:Theta}
yields estimate \eqref{est-oscillation}.
%\begin{align*}
%\fint_{B_{3\sigma}} |\nabla (f-v)|^p\, dx 
%\leq C\,\, \Theta_{B_{3\sigma}}(\A)^{\frac{p_0 - p}{p_0}}.
%\end{align*}
We next consider 
the case $1<p<2$. Then   condition (1.2) in \cite{NP} is satisfied thanks to \eqref{structural-consequence}.
Therefore, instead of \eqref{p-geq-2}  we now 
 have from  \cite[Lemma~3.1]{NP} that 
\begin{equation*}
\int_{B_{3\sigma}} |\nabla (f-v)|^p\, dx 
\leq  \tau \int_{B_{3\sigma}} |\nabla f|^p\, dx + 
C_p  \tau^{1-\frac2p} I \quad\mbox{for all}\quad \tau\in (0,\frac12). 
\end{equation*}
Then we deduce from  estimate \eqref{est-I} for $I$ and Young's inequality that 
\begin{align*}
\fint_{B_{3\sigma}} |\nabla (f-v)|^p\, dx 
&\leq 
 2 \tau \int_{B_{3\sigma}} |\nabla f|^p\, dx
+ C(p,\Lambda)   \tau^{(1-\frac2p)p'}\fint_{B_{3\sigma}}\Big[ \sup_{\xi\neq 0} \frac{|\ba(x,\xi) -\bar\ba_{B_{3\sigma}}(\xi)| }{|\xi|^{p-1}}\Big]^{p'}\, |\nabla f|^p  dx.
\end{align*}
The first integral is estimated by \eqref{est-nabla-f}
 and  the last integral can be estimated exactly as above. As a consequence, we obtain 
\begin{align*}
\fint_{B_{3\sigma}} |\nabla (f-v)|^p\, dx 
&\leq 
 2^{p+1} \big(\frac43\big)^n \,  \tau 
+ C(p,n,\Lambda)   \tau^{(1-\frac2p)p'}\Theta_{B_{3\sigma}}(\A)^{\frac{p_0 - p}{p_0}}\quad\mbox{for all}\quad \tau\in (0,\frac12). 
\end{align*}
\end{proof}

We are now ready to prove Proposition~\ref{lm:localized-compare-gradient}.
\begin{proof}[\bf Proof of Proposition~\ref{lm:localized-compare-gradient}]
The stated result follows by scaling and then applying
Lemma~\ref{lm:step1} and Lemma~\ref{lm:step2}. Precisely, let us define
\[
\A'(x, z, \xi) = \A(r x, z,\xi),\quad \F'(x)=\F(r x), \quad u'(x) = r^{-1} u(r x),\quad\mbox{and }\, \,  \theta' = \theta r. 
\]
Then $u'$ is a weak solution of  $\div \Big[\frac{\A'(x,\lambda\theta' u', \lambda \nabla u')}{\lambda^{p-1}}\Big] = \div \F'\,$ in $B_{4}$ satisfying
\begin{equation*}
 \|u'\|_{L^\infty(B_{4})}\leq \frac{M_0}{\lambda\theta'} ,\quad \fint_{B_{4}}{|\nabla  u'|^p \, dx
 =\fint_{B_{4 r}}{|\nabla  u|^p \, dy}}\leq 1,\quad
  \fint_{B_{4\sigma }}{|\nabla  u'|^p \, dx
 =\fint_{B_{4\sigma  r}}{|\nabla  u|^p \, dy}}\leq 1.
\end{equation*}
We also have
\begin{equation*}
 \Theta_{B_{3\sigma}}(\A')=\Theta_{B_{3\sigma r}}(\A)\leq \delta  \quad\mbox{and}\quad\fint_{B_{4}}{|\F'|^{p'} \, dx
 =\fint_{B_{4 r}}{|\F|^{p'} \, dy}}\leq \delta.
\end{equation*}
Therefore, we can apply Lemma~\ref{lm:step1} and Lemma~\ref{lm:step2} to obtain that
\begin{equation}\label{u'-v'}
\fint_{B_{3\sigma}} |\nabla u' -\nabla v'|^pdx \leq \e^p,
\end{equation}
where $v'$ is a weak solution of 
\[
\div \Big[\frac{\bar \A'_{B_\sigma}(\lambda\theta' \bar h_{B_{4\sigma}}, \lambda \nabla v' )}{\lambda^{p-1}}\Big] =0 \quad \mbox{in}\quad B_{3 \sigma}.
\]
From the well known interior Lipschitz estimate for this constant coefficient equation  we have
\begin{equation}\label{nabla-v'}
\|\nabla v' \|_{L^\infty(B_{2\sigma})}^p\leq C(p,n,\Lambda) \fint_{B_{3\sigma}} |\nabla v'|^p\, dx.
\end{equation}
Let $v(y):= r v'(r^{-1} y)$. Then by rescaling we obtain from \eqref{u'-v'} and \eqref{nabla-v'} that
\begin{equation*}
\fint_{B_{3 \sigma r}} |\nabla u -\nabla v|^pdy \leq \e^p\quad \mbox{and}\quad
\|\nabla v \|_{L^\infty(B_{2\sigma r})}^p\leq C(p,n,\Lambda) \fint_{B_{3 \sigma r}} |\nabla v|^p\, dy.
\end{equation*}

\end{proof}

\section{Density and gradient estimates}\label{interior-density-gradient}
We derive interior  gradient estimates for weak solution $u$ of \eqref{ME} by estimating the distribution functions of the maximal
function of $|\nabla  u|^p$. This is carried out in the next two subsections, while the last subsection (Subsection~\ref{sec:Morrey-Spaces}) 
is devoted to proving the main results stated in Section~\ref{sec:Intro}.
\subsection{Density estimates}\label{sub:density-est}
The next result gives a density estimate for the distribution  of $\M_{B_1}(|\nabla  u|^p)$. It roughly says that  if the maximal 
function $\M_{B_1}(|\nabla  u|^p)$ is bounded at one point in $B_{\sigma r}(y)$ then this property can be propagated for all points in $B_{\sigma r}(y)$ 
except on a set of small measure $w$. 
\begin{lemma}\label{initial-density-est}
Assume that   $\A$ satisfies \eqref{structural-reference-1}--\eqref{structural-reference-3} and  $\F\in L^{p'}(B_{10};\R^n)$.
Let  $M_0>0$  and  $w$ be an $A_\infty$ weight.   There exists a constant $N=N(p, n,\Lambda)>1$ satisfying  
for  any $\e>0$, we can find   
small positive constants $\delta$  and $\sigma$ depending only on $\e,\, p,\, n,\, \omega,\, \Lambda,\, M_0$, and $[w]_{A_\infty}$
such that:  if $\lambda>0$, $\theta>0$,
\begin{equation*}\label{interior-SMO}
\sup_{0<\rho\leq \frac35}\sup_{y\in B_{\frac{\sigma}{10}}} \Theta_{B_{\sigma\rho}(y)}(\A)\leq  \delta,
\end{equation*}
then
% for any $\A\in \calA_{\delta, 4}
  for any weak solution $u$ of \eqref{ME} with
 $\|u\|_{L^\infty(B_1)}\leq \frac{M_0}{\lambda \theta}$,
and for any $y\in B_{\frac{\sigma}{10}}$, $0<r\leq \frac15$ with
\begin{equation}\label{one-point-condition}
 B_{\sigma r}(y)\cap B_{\frac{\sigma}{10}}\cap \big\{B_1:\, \M_{B_1}(|\nabla  u|^p)\leq 1 \big\}\cap \{ B_1: \M_{B_1}(|\F|^{p'} )\leq \delta\}\neq \emptyset,
\end{equation}
we have
\begin{equation*}
 w\Big( \{ B_{\frac{\sigma}{10}}:\, \M_{B_1}(|\nabla  u|^p)> N \}\cap B_{\sigma r}(y)\Big)
< \e  \, w(B_{\sigma r}(y)).
\end{equation*}
\end{lemma}
\begin{proof}
By \eqref{one-point-condition} there exists $x_0\in B_{\sigma r}(y)\cap B_{\frac{\sigma}{10}}$ such that 
\begin{align}\label{max-one-point}
\M_{B_1}(|\nabla  u|^p)(x_0) \leq 1 \quad \mbox{and}\quad  \M_{B_1}(|\F|^{p'} )(x_0)\leq \delta.
\end{align}
This together with the facts $B_{4 r}(y)\subset B_{5 r}(x_0)\cap B_1$ and $B_{4 \sigma r}(y)\subset B_{5 \sigma r}(x_0)\cap B_1$ implies that
\begin{align*}
&\fint_{B_{4 r}(y)}{|\nabla  u|^p dx}
\leq \frac{|B_{5 r}(x_0)|}{|B_{4 r}(y)|} \frac{1}{|B_{5 r}(x_0)|} \int_{B_{5 r}(x_0)\cap B_1}{|\nabla  u|^p \, dx}\leq \big(\frac54\big)^n,\\
&\fint_{B_{4\sigma r}(y)}{|\nabla  u|^p dx}
\leq \frac{|B_{5\sigma r}(x_0)|}{|B_{4\sigma r}(y)|} \frac{1}{|B_{5 \sigma r}(x_0)|} \int_{B_{5\sigma r}(x_0)\cap B_1}{|\nabla  u|^p \, dx}\leq \big(\frac54\big)^n,\\
  &\fint_{B_{4r}(y)} |\F|^{p'} \, dx
 \leq \frac{|B_{5 r}(x_0)|}{|B_{4 r}(y)|} \frac{1}{|B_{5 r}(x_0)|} \int_{B_{5 r}(x_0)\cap B_1}{|\F|^{p'}  dx}\leq \big(\frac54\big)^n \delta.
\end{align*}
Therefore, we can apply  Proposition~\ref{lm:localized-compare-gradient} and Remark~\ref{rm:translation-invariant} for 
$\tilde \e \in (0,1]$ that will be determined later.
As a consequence, we obtain 
\begin{equation}\label{eq:nabla-u-v}
\fint_{B_{3\sigma r}(y)}{|\nabla  u - \nabla  v|^p\, dx}\leq \tilde\e^p
\end{equation}
for some function $v\in  W^{1,p}(B_{3 \sigma r}(y))$  with
\begin{equation*}
\|\nabla  v\|_{L^\infty(B_{2 \sigma  r}(y))}^p
\leq C(p,n,\Lambda) \fint_{B_{3 \sigma r}(y)}{|\nabla  v|^p \, dx}.
\end{equation*}
These combined with the first inequality in \eqref{max-one-point} give
\begin{equation}\label{bounded-grad}
\|\nabla  v\|_{L^\infty(B_{2 \sigma  r}(y))}^p\leq 2^{p-1} C \left[ \fint_{B_{3\sigma r}(y)}{|\nabla  u - \nabla  v|^p\, dx} +
\fint_{B_{3\sigma r}(y)}{|\nabla  u |^p\, dx}\right]\leq 2^{p-1} C \Big[ \tilde\e^p+
\big(\frac43\big)^n \Big]\leq C_*,
\end{equation}
where $C^*>0$ depends only on $p$, $n$, and $\Lambda$.
We claim that \eqref{max-one-point}--\eqref{bounded-grad} yield
\begin{equation}\label{key-claim}
\big\{B_{\sigma r}(y): \, \M_{B_{3\sigma r}(y)}(|\nabla u - \nabla v|^p) \leq C_*\big\} 
\subset \big\{B_{\sigma r}(y):\, \M_{B_1}(|\nabla  u|^p)\leq  N \big\}
\end{equation}
with $N := \max{\{2^{p} C_*, 3^n\}}$.
Indeed, let $x$ be a point in the set on the left hand side of \eqref{key-claim}, and consider $B_\rho(x)$. If $\rho\leq \sigma r$, 
then $B_\rho(x)\subset B_{2\sigma r}(y)\subset B_1$ and hence 
\begin{align*}
 \frac{1}{|B_{\rho}(x)|} \int_{B_{\rho}(x)\cap B_1}{|\nabla  u|^p  dy}
 &\leq \frac{2^{p-1}}{|B_{\rho}(x)|}\Big[ \int_{B_{\rho}(x)\cap B_1}{|\nabla  u -\nabla v|^p dy}
 +\int_{B_{\rho}(x)\cap B_1}{|\nabla  v|^p dy} \Big]\\
 &\leq 2^{p-1} \Big[\M_{B_{3\sigma r}(y)}(|\nabla u - \nabla v|^p)(x)
 +\|\nabla  v\|_{L^\infty(B_{2 \sigma  r}(y))}^p  \Big]\leq 2^p C_*.
\end{align*}
On the other hand, if $\rho> \sigma r$ then $B_\rho(x)\subset B_{3\rho}(x_0)$.
This and the first inequality in \eqref{max-one-point} give
\begin{align*}
 \frac{1}{|B_{\rho}(x)|} \int_{B_{\rho}(x)\cap B_1}{|\nabla  u|^p  dy}
 &\leq  \frac{3^n}{|B_{3\rho}(x_0)|} \int_{B_{3 \rho}(x_0)\cap B_1}{|\nabla  u|^p  dy}\leq 3^n.
\end{align*}
Therefore, $\M_{B_1}(|\nabla  u|^p)(x)\leq  N $ and  claim \eqref{key-claim} is proved. Notice that \eqref{key-claim} is equivalent to 
\begin{equation*}
 \big\{B_{\sigma r}(y):\, \M_{B_1}(|\nabla  u|^p)>  N \big\} \subset \big\{B_{\sigma r}(y): \, \M_{B_{3\sigma r}(y)}(|\nabla u - \nabla v|^p) > C_*\big\}. 
\end{equation*}
It follows from this, the weak type $1-1$ estimate, and \eqref{eq:nabla-u-v} that
\begin{align*}
 \big|\big\{B_{\sigma r}(y):\, \M_{B_1}(|\nabla  u|^p)>  N \big\}\big|
& \leq \big| \big\{B_{\sigma r}(y): \, \M_{B_{3\sigma r}(y)}(|\nabla u - \nabla v|^p) > C_*\big\} \big|\\
&\leq C \int_{B_{3\sigma r}(y)} |\nabla u - \nabla v|^p dx\leq C_1 \tilde\e^p |B_{\sigma r}(y)|.
\end{align*}
We then infer  from property \eqref{charac-A-infty} that 
\begin{align*}
 w\Big(\big\{B_{\sigma r}(y):\, \M_{B_1}(|\nabla  u|^p)>  N \big\}\Big)
 &\leq   A \Big(\frac{| \big\{B_{\sigma r}(y):\, \M_{B_1}(|\nabla  u|^p)>  N \big\}|}{|B_{\sigma r}(y)|}\Big)^{\nu}
w(B_{\sigma r}(y))
\leq A (C_1\tilde\e^{p})^\nu w(B_{\sigma r}(y))
\end{align*}
with   $A$ and $\nu$ being  the  constants given by  characterization \eqref{charac-A-infty} for $w$. We choose  $\tilde \e^p  := \min{\{C_1^{-1}(\e A^{-1})^{\frac{1}{\nu}},1\}}$ to complete the proof.
\end{proof}

In view of Lemma~\ref{initial-density-est}, we can apply a variation of the Vitali covering lemma as in \cite[Lemma~3.8]{MP2} 
%for \[
%C=\{ B_1:\, \M_{B_5}(|%\nabla   u|^p)> N \} \quad\mbox{and}\quad  
%D=\{ B_1:\, \M_{B_5}(|\nabla  u|^p)> 1 \}\cup \{ B_1:\, \M_{B_5}(|\F|^{\frac{p}{p-1}})> \delta \}
%\]
 to obtain:
\begin{lemma}\label{second-density-est}
Assume that   $\A$ satisfies \eqref{structural-reference-1}--\eqref{structural-reference-3} and  $\F\in L^{p'}(B_{10};\R^n)$.
Let  $M_0>0$  and  $w$ be an $A_s$ weight for some $1<s<\infty$.
There exists a constant $N=N(p, n,\Lambda)>1$ satisfying  
for  any $\e>0$, we can find   
small positive constants $\delta$  and $\sigma$ depending only on $\e,\, p,\, n,\, \omega,\,\Lambda,\, M_0$, and $[w]_{A_s}$
such that:  if $\lambda>0$, $\theta>0$,
\begin{equation*}
\sup_{0<\rho\leq \frac35}\sup_{y\in B_{\frac{\sigma}{10}}} \Theta_{B_{ \sigma\rho }(y)}(\A)\leq  \delta,
\end{equation*}
then   for any weak solution $u$ of \eqref{ME} satisfying $\|u\|_{L^\infty(B_1)}\leq \frac{M_0}{\lambda \theta}$ and 
\begin{equation}\label{density-ass}
  w \Big( \{B_{\frac{\sigma}{10}}: 
\M_{B_1}(|\nabla  u|^p)> N \}\Big) <\e \, w(B_{\frac{\sigma}{10}}),
\end{equation}
we have
\begin{align*}
w \Big(\{B_{\frac{\sigma}{10}}: \M_{B_1}(|\nabla u|^p)> N\}\Big)
\leq 10^{ns } [w]_{A_s}^2\e \, \Big[
w\Big(\{B_{\frac{\sigma}{10}}: \M_{B_1}(|\nabla u|^p)> 1\}\Big)
+w\Big(\{ B_{\frac{\sigma}{10}}: \M_{B_1}(|\F|^{p'})> \delta \}\Big)\Big].\nonumber
\end{align*}
\end{lemma}
%\begin{remark}
%In the proof below,  we only need to use the first inequality in \eqref{strong-doubling} for $w$ which is the strong doubling property of $w%$.
%\end{remark}
\begin{proof}
%We include the arguments here  %for the sake of completeness. 
Set
\[
C=\{ B_{\frac{\sigma}{10}}:\, \M_{B_1}(|\nabla   u|^p)> N \} \quad\mbox{and}\quad  
D=\{ B_{\frac{\sigma}{10}}:\, \M_{B_1}(|\nabla  u|^p)> 1 \}\cup \{B_{\frac{\sigma}{10}}:\, \M_{B_1}(|\F|^{p'})> \delta \}.
\]
Let $y$ be any point in $C$, and define 
\[
m(r):= \frac{w(C\cap B_{\sigma r}(y))}{w(B_{\sigma r}(y))} \quad\mbox{for}\quad  r>0.
\]
 Due to the lower semicontinuity of $\M_{B_1}(|\nabla   u|^p)$, we have   $\lim_{r\to 0^+} m(r)= 1$ as $C$ is open. Moreover, condition  \eqref{density-ass} implies that
$
 m(r) \leq \frac{w(C)}{w(B_{\frac{\sigma}{10}})} <\e$  when $ r\geq \frac15.
$
Hence, there exists $r_y\in (0,\frac15)$ such that
$ m(r_y)=\e$ and $m(r)<\e $ for all $r>r_y$. That is,
\begin{equation}\label{density-choice}
w(C\cap B_{\sigma r_y}(y))=\e\, w(B_{\sigma r_y}(y))\quad \mbox{and}\quad w(C\cap B_{\sigma r}(y))< \e \, w(B_{\sigma r}(y))\quad \forall r>r_y.
\end{equation}
Therefore, by Vitali's covering lemma we can select a countable sequence $\{y_i\}_{i=1}^\infty$ such that $\{B_{\sigma r_i}(y_i)\}$ is a sequence of  disjoint balls and
\begin{equation}\label{covering}
C\subset \bigcup_{i=1}^\infty B_{5\sigma  r_i}(y_i),
\end{equation}
where $r_i := r_{y_i}$. Since 
$w(C\cap B_{\sigma r_i}(y_i))=\e\, w(B_{\sigma r_i}(y_i))$ by \eqref{density-choice} and $[w]_{A_\infty}\leq [w]_{A_s}$ by Lemma~\ref{weight:basic-pro}, it follows from Lemma~\ref{initial-density-est} that
\begin{equation}\label{inD}
B_{\sigma r_i}(y_i)\cap B_{\frac{\sigma}{10}}\subset D.
\end{equation}
Using \eqref{density-choice}--\eqref{covering} and the first inequality in \eqref{strong-doubling}, we have 
\begin{align}\label{w-Cballs}
w(C) 
&\leq w\Big(\bigcup_{i=1}^\infty  B_{5 \sigma r_i}(y_i) \cap C\Big)\leq \sum_{i=1}^\infty w(B_{5\sigma  r_i}(y_i) \cap C)\nonumber\\
&\leq \e  \sum_{i=1}^\infty w(B_{5\sigma  r_i}(y_i) )\leq \e [w]_{A_s}  5^{ n s}\sum_{i=1}^\infty w(B_{\sigma  r_i}(y_i) ).
\end{align}
We claim that
\begin{equation}\label{simple-geometry}
\sup_{y\in  B_{\frac{\sigma}{10}}, \, 0<r<\frac15}\frac{|B_{\sigma r}(y)|}{|B_{\sigma r}(y)\cap B_{\frac{\sigma}{10}}|} 
\leq 2^n,
\end{equation}
which together with \eqref{strong-doubling} gives 
$
w(B_{\sigma r_i}(y_i))\leq [w]_{A_s}  2^{n s} \, w(B_{\sigma r_i}(y_i)\cap B_{\frac{\sigma}{10}})$.
It follows from  this and    \eqref{inD}--\eqref{w-Cballs} that
\[
w(C) \leq  \e [w]_{A_s}^2  10^{ n s}\sum_{i=1}^\infty w(B_{ \sigma r_i}(y_i) \cap B_{\frac{\sigma}{10}})= 
\e [w]_{A_s}^2  10^{ n s} w\Big(\bigcup_{i=1}^\infty B_{\sigma  r_i}(y_i) \cap B_{\frac{\sigma}{10}} \Big)\leq \e [w]_{A_s}^2  10^{ n s} w(D),
\]
which implies the desired estimate. Thus it remains to show \eqref{simple-geometry}. For this, let 
$y\in  B_{\frac{\sigma}{10}}$, $r\in (0, 1/5)$, and it is enough  to consider only the case $y\neq 0$.
%and $B_{\sigma r}(y)\not\subset B_{\frac{\sigma}{10}}$.
Then the line  passing through $y$ and $0$ intersects  $\partial B_{\sigma r}(y)$ at two distinct points, say $a_1$ and $a_2$ with $a_1$ 
being the one on the same side as the origin $0$ with respect to the point $y$. If $a_1\not\in B_{\frac{\sigma}{10}}$, 
then $B_{\frac{\sigma}{10}}\subset B_{\sigma r}(y)$ since for any $x\in B_{\frac{\sigma}{10}}$ we have  
$|x-y| \leq |x| +|y|=|x| + |y-a_1|-|a_1|< \frac{\sigma}{10}+ \sigma r -\frac{\sigma}{10}=\sigma r$. It follows that
\[
\frac{|B_{\sigma r}(y)|}{|B_{\sigma r}(y)\cap B_{\frac{\sigma}{10}}|} 
= \frac{|B_{\sigma r}(y)|}{| B_{\frac{\sigma}{10}}|} =(10r)^n < 2^n.
\]
 On the other hand if $a_1\in B_{\frac{\sigma}{10}}$, then by letting $z$ be the midpoint of $a_1$ and $y$ we obviously have 
   $B_{\frac{\sigma r}{2}}(z) \subset B_{\sigma r}(y)$. If $0$ belongs to the line segment $[z,y]$ connecting $z$ and $y$, then  $|z|=|a_1|-|a_1-z|< \frac{\sigma}{10} -\frac{\sigma r}{2}$. In case  $0\not\in  [z,y]$, then    $|z|=|y|-|y-z|< \frac{\sigma}{10} -\frac{\sigma r}{2}$. Thus we always have $|z|< \frac{\sigma}{10} -\frac{\sigma r}{2}$, and  hence 
       $B_{\frac{\sigma r}{2}}(z) \subset  B_{\frac{\sigma}{10}}$ as  $x\in B_{\frac{\sigma r}{2}}(z)$ implies that
   $|x| \leq |x - z| +|z|<\frac{\sigma r}{2} +\frac{\sigma}{10} -\frac{\sigma r}{2}=\frac{\sigma}{10} $. Therefore, we deduce that
$B_{\frac{\sigma r}{2}}(z) \subset B_{\sigma r}(y)\cap B_{\frac{\sigma}{10}}$ giving
\[
\frac{|B_{\sigma r}(y)|}{|B_{\sigma r}(y)\cap B_{\frac{\sigma}{10}}|} 
\leq \frac{|B_{\sigma r}(y)|}{| B_{\frac{\sigma r}{2}}(z)|} =2^n.
\]
We then conclude that \eqref{simple-geometry} holds, and the proof is complete.
\end{proof}

\subsection{Gradient estimates in weighted $L^q$ spaces}\label{sub:Lebesgue-Spaces}

We are now ready to prove  Theorem~\ref{thm:main}.

\begin{proof}[\textbf{Proof of Theorem~\ref{thm:main}}]
Let  $N=N(p,n,\Lambda)>1$ be as in  Lemma~\ref{second-density-est}, and let $l=q/p\geq 1$. We choose  
$\e=\e(p, q, n,\Lambda,s, [w]_{A_s} )>0$ be such that
\[
\e_1 \eqdef 10^{n s } [w]_{A_s}^2\e = \frac{1}{2 N^{l}}, 
\]
and let $\delta$  and $\sigma$ (depending only on $p,\,q,\,n,\,\omega, \,\Lambda,\, M_0,\, s$, and $[w]_{A_s}$) be the corresponding 
positive constants given by Lemma~\ref{second-density-est}.
Assuming for a moment that $u$  satisfies 
\begin{equation}\label{initial-distribution-condition}
w\big( \{B_{\frac{\sigma}{10}}: 
\M_{B_1}(|\nabla u|^p)> N \}\big) <  \e \, w(B_{\frac{\sigma}{10}}).
\end{equation}
Then it follows from  Lemma~\ref{second-density-est} that
\beq\label{initial-distribution-est}
w\big(\{B_{\frac{\sigma}{10}}: \M_{B_1}(|\nabla u|^p)> N\}\big)
\leq \e_1  \left[
w\big(\{B_{\frac{\sigma}{10}}: \M_{B_1}(|\nabla u|^p)> 1\}\big)
+ w\big(\{ B_{\frac{\sigma}{10}}: \M_{B_1}(|\F |^{p'})> \delta \}\big)\right].
\eeq
 Let us iterate this estimate by considering
\[
u_1(x) = \frac{u(x)}{N^{\frac1p}}, \quad \F_1(x) = \frac{\F(x)}{N^{\frac{p-1}{p}}}\quad \mbox{and}\quad \lambda_1 = N^{\frac1p}\lambda.
\]
It is clear that $\|u_1\|_{L^\infty(B_1)}\leq \frac{M_0}{\lambda_1 \theta}$ and  $u_1$    is a weak solution of
$
\div \Big[\frac{\A(x,\lambda_1 \theta u_1, \lambda_1 \nabla u_1)}{\lambda_1^{p-1}}  \Big]=\div \F_1$ in   $B_{10}
$.
%\begin{equation*}
%\div \Big[\frac{\A(x,\lambda_1 \theta u_1, \lambda_1 \nabla u_1)}{\lambda_1^{p-1}}  \Big]=\div  \F_1  \quad\mbox{in}\quad B_{2}.
%\end{equation*}
Moreover, thanks to \eqref{initial-distribution-condition} we have
\begin{align*}
w\big( \{B_{\frac{\sigma}{10}}: 
\M_{B_1}(|\nabla u_1|^p)> N \}\big) &= w\big( \{B_{\frac{\sigma}{10}}: 
\M_{B_1}(|\nabla u|^p)> N^2 \}\big) < \e \, w(B_{\frac{\sigma}{10}}).
\end{align*}
Therefore, by applying  Lemma~\ref{second-density-est} to $u_1$ we obtain
\begin{align*}
w\big(\{B_{\frac{\sigma}{10}}: \M_{B_1}(|\nabla u_1|^p)> N\}\big)
&\leq \e_1 \left[
w\big(\{B_{\frac{\sigma}{10}}: \M_{B_1}(|\nabla u_1|^p)> 1 \}\big)
+ w\big(\{ B_{\frac{\sigma}{10}}: \M_{B_1}(|\F_1 |^{p'} )> \delta \}\big)\right]\\
&= \e_1  \left[ 
w\big(\{B_{\frac{\sigma}{10}}: \M_{B_1}(|\nabla u|^p)> N \}\big)
+ w\big(\{ B_{\frac{\sigma}{10}}: \M_{B_1}(|\F|^{p'})> \delta N\}\big) \right].
\end{align*}
We infer from this and  \eqref{initial-distribution-est} that
\begin{align}\label{first-iteration-est}
&w\big(\{B_{\frac{\sigma}{10}}: \M_{B_1}(|\nabla u|^p)> N^2\}\big)
\leq \e_1^2 
w\big(\{B_{\frac{\sigma}{10}}: \M_{B_1}(|\nabla u|^p)> 1 \}\big)\\
&\qquad + \e_1^2 w\big(\{ B_{\frac{\sigma}{10}}: \M_{B_1}(|\F|^{p'})> \delta\} \big)+ \e_1w\big(\{ B_{\frac{\sigma}{10}}: 
\M_{B_1}(|\F|^{p'})> \delta N\}\big). \nonumber
\end{align}
 Next, let
\[
u_2(x) = \frac{u(x)}{N^{\frac2p}}, \quad \F_2(x) = \frac{\F(x)}{N^{\frac{2(p-1)}{p}}} \quad \mbox{and}\quad \lambda_2 = N^{\frac2p}\lambda.
\]
%\marginpar{simplify\\ step $u_2$?}
Then  $u_2$  is a weak solution of
$
\div \Big[\frac{\A(x,\lambda_2 \theta u_2, \lambda_2 \nabla u_2)}{\lambda_2^{p-1}}  \Big]=\div \F_2$ in   $B_{10}
$
satisfying $\|u_2\|_{L^\infty(B_1)}\leq M_0/ (\lambda_2 \theta) $ and 
\begin{align*}
  w\big( \{B_{\frac{\sigma}{10}}: 
\M_{B_1}(|\nabla u_2|^p)> N \}\big) &= w\big( \{B_{\frac{\sigma}{10}}: 
\M_{B_1}(|\nabla u|^p)> N^3 \}\big)<\e \,  w(B_{\frac{\sigma}{10}}).
\end{align*}
Hence by applying  Lemma~\ref{second-density-est} to $u_2$ we get
\begin{align*}
w\big(\{B_{\frac{\sigma}{10}}: \M_{B_1}(|\nabla u_2|^p)> N\}\big)
&\leq \e_1 \left[ 
w\big(\{B_{\frac{\sigma}{10}}: \M_{B_1}(|\nabla u_2|^p)> 1 \}\big)
+ w\big(\{ B_{\frac{\sigma}{10}}: \M_{B_1}(|\F_2|^{p'} )> \delta\}\big)\right]\\
&= \e_1  \left[ 
w\big(\{B_{\frac{\sigma}{10}}: \M_{B_1}(|\nabla u|^p)> N^2 \}\big)
+ w\big(\{ B_{\frac{\sigma}{10}}: \M_{B_1}(|\F|^{p'})> \delta N^2\}\big) \right].
\end{align*}
This together with  \eqref{first-iteration-est} gives
\begin{align*}
w\big(\{B_{\frac{\sigma}{10}}: \M_{B_1}(|\nabla u|^p)> N^3\}\big)
\leq \e_1^3 
w\big(\{B_{\frac{\sigma}{10}}: \M_{B_1}(|\nabla u|^p)> 1 \}\big)
+ \sum_{i=1}^3\e_1^i w\big(\{ B_{\frac{\sigma}{10}}: \M_{B_1}(|\F|^{p'})> \delta N^{3-i}\} \big).
\end{align*}
By repeating the iteration, we then conclude that
\begin{align*}
w\big(\{B_{\frac{\sigma}{10}}: \M_{B_1}(|\nabla u|^p)> N^k\}\big)
\leq \e_1^k 
w\big(\{B_{\frac{\sigma}{10}}: \M_{B_1}(|\nabla u|^p)> 1 \}\big)
 + \sum_{i=1}^k\e_1^iw\big(\{ B_{\frac{\sigma}{10}}: \M_{B_1}(|\F|^{p'})> \delta N^{k-i}\} \big)
\end{align*}
for all $k=1,2,\dots$ This together with  
\begin{align*}
&\int_{B_{\frac{\sigma}{10}}}\M_{B_1}(|\nabla u|^p)^{l} \, dw  =l \int_0^\infty t^{l-1} w\big(\{B_{\frac{\sigma}{10}}: \M_{B_1}(|\nabla u|^p)>t\}\big)\, dt\\
&=l \int_0^{N} t^{l -1} w\big(\{B_{\frac{\sigma}{10}}: \M_{B_1}(|\nabla u|^p)>t\}\big)\, dt
+l \sum_{k=1}^\infty\int_{N^{k}}^{N^{k+1}} t^{l -1} w\big(\{B_{\frac{\sigma}{10}}: \M_{B_1}(|\nabla u|^p)>t\}\big)\, dt\\
&\leq N^{l} w(B_{\frac{\sigma}{10}}) + (N^{l} -1) \sum_{k=1}^\infty N^{l k}w\big(\{B_{\frac{\sigma}{10}}: \M_{B_1}(|\nabla u|^p)>N^k\}\big)
\end{align*}
gives
\begin{align*}
\int_{B_{\frac{\sigma}{10}}}\M_{B_1}(|\nabla u|^p)^{l} \, dw  
&\leq  N^{l} w(B_{\frac{\sigma}{10}}) + (N^{l} -1)w(B_{\frac{\sigma}{10}}) \sum_{k=1}^\infty (\e_1 N^{l})^k\\
&\quad +\sum_{k=1}^\infty\sum_{i=1}^k (N^{l} -1)N^{l k}\e_1^i w\big(\{ B_{\frac{\sigma}{10}}: \M_{B_1}(|\F|^{p'})> \delta N^{k-i}\} \big).
\end{align*}
But we have
\begin{align*}
&\sum_{k=1}^\infty\sum_{i=1}^k (N^{l} -1)N^{{l} k}\e_1^iw\big(\{ B_{\frac{\sigma}{10}}: 
\M_{B_1}(|\F|^{p'})> \delta N^{k-i}\} \big)\\
&=\big(\frac{N}{\delta}\big)^{l}\sum_{i=1}^\infty (\e_1 N^{l})^i\left[\sum_{k=i}^\infty 
(N^{l} -1)\delta^{{l}}N^{{l} (k-i-1)} w\big(\{B_{\frac{\sigma}{10}}: \M_{B_1}(|\F|^{p'})>
\delta N^{k-i}\} \big)\right]\\
&=\big(\frac{N}{\delta}\big)^{l}\sum_{i=1}^\infty (\e_1 N^{l})^i\left[\sum_{j=0}^\infty (N^{l} -1)\delta^{{l}}N^{{l} (j-1)}
w\big(\{ B_{\frac{\sigma}{10}}: \M_{B_1}(|\F|^{p'})> \delta N^{j}\} \big)\right]\\
&\leq \big(\frac{N}{\delta}\big)^{l} \Bigg[\int_{B_{\frac{\sigma}{10}}}\M_{B_1}(|\F|^{p'})^{l}  \, dw \Bigg]\sum_{i=1}^\infty (\e_1 N^{l})^i.
\end{align*}
Thus we infer that
\begin{align*}
\int_{B_{\frac{\sigma}{10}}}\M_{B_1}(|\nabla u|^p)^{l} \, dw 
&\leq N^{l} w(B_{\frac{\sigma}{10}}) + \left[ (N^{l} -1) w(B_{\frac{\sigma}{10}}) +\big(\frac{N}{\delta}\big)^{l} 
\int_{B_{\frac{\sigma}{10}}}\M_{B_1}(|\F|^{p'})^{l}  \, dw  \right] \sum_{k=1}^\infty (\e_1 N^{l})^k\\
&= N^{l} w(B_{\frac{\sigma}{10}})+ \left[ (N^{l} -1)w(B_{\frac{\sigma}{10}}) +\big(\frac{N}{\delta}\big)^{l}
\int_{B_{\frac{\sigma}{10}}}\M_{B_1}(|\F|^{p'})^{l}  \, dw \right] \sum_{k=1}^\infty 2^{-k}\\
&\leq C\left( w(B_{\frac{\sigma}{10}})+ \int_{B_{\frac{\sigma}{10}}}\M_{B_1}(|\F|^{p'})^{l} \, dw  \right)
\end{align*}
with the constant $C$ depending only on $p$, $q$, $n$, $\omega$, $ \Lambda$,    $M_0$, $s$, and $[w]_{A_s}$. This together with the facts $l =q/p$ and  $
|\nabla u(x)|^p  \leq  
\M_{B_1}(|\nabla u|^p)(x)$
for a.e.  $x\in B_{\frac{\sigma}{10}}$ yields
\begin{equation}\label{initial-desired-L^p-estimate}
\fint_{B_{\frac{\sigma}{10}}}|\nabla u|^q dw \leq C\left( 1+ \fint_{B_{\frac{\sigma}{10}}}\M_{B_1}(|\F|^{p'})^{\frac{q}{p}} \, dw  \right).
\end{equation}

We next remove the extra assumption  
\eqref{initial-distribution-condition} for $u$. Notice that for any $M>0$, by using the weak type $1-1$ estimate for the maximal function  we get
\begin{align}\label{weak1-1}
\big| \{B_{\frac{\sigma}{10}}: 
\M_{B_1}(|\nabla u|^p)> N  M^p\}\big|
\leq \frac{C }{N M^p}\int_{B_1} |\nabla u|^p dx.
\end{align}
Let $\bar{u}(x,t) = u(x,t)/ M$, where
\[ M^p := \frac{2 C   \|\nabla u\|_{L^p(B_1)}^p }{N (\e K^{-1})^{\frac{1}{\beta}} |B_{\frac{\sigma}{10}}|}
\]
with
$K$ and $\beta$ being the constants given by  \eqref{strong-doubling} and depending only on $n$ and $[w]_{A_s}$. 
Then it follows from \eqref{weak1-1} that
\[\big| \{B_{\frac{\sigma}{10}}: 
\M_{B_1}(|\nabla \bar{u}|^p)> N\}\big|
\leq  \frac12 (\e K^{-1})^{\frac{1}{\beta}} |B_{\frac{\sigma}{10}}|.
\]
This together with the second inequality in
 \eqref{strong-doubling}   gives
\[w \big( \{B_{\frac{\sigma}{10}}: 
\M_{B_1}(|\nabla \bar{u}|^p)> N\}\big)
< \e \, w(B_{\frac{\sigma}{10}}).
\]
 Hence we can apply \eqref{initial-desired-L^p-estimate} to $\bar{u}$ with $\F$ and $\lambda$ being replaced by 
$\bar \F= \F/M^{p-1}$ and $\bar \lambda = \lambda M$. By reversing back to the  functions $u$ and  $\F$,  we then  obtain 
\begin{align*}
\fint_{B_{\frac{\sigma}{10}}}|\nabla u|^q dw 
&\leq C\left( M^q+ \fint_{B_{\frac{\sigma}{10}}}\M_{B_1}(|\F|^{p'})^{\frac{q}{p}} \, dw \right)
\leq  C\left( \|\nabla u\|_{L^p(B_1)}^q+ \fint_{B_{\frac{\sigma}{10}}}\M_{B_1}(|\F|^{p'})^{\frac{q}{p}} \, dw  \right).
\end{align*}
Due to the translation invariance of our equation, the above estimate holds true if $B_{\frac{\sigma}{10}}$ and  $B_1$ 
are respectively  replaced by   $B_{\frac{\sigma}{10}}(z)\subset B_\frac32$ and  $B_1(z) \subset B_2$
for any $z\in B_1$. Therefore,  desired estimate 
\eqref{main-estimate} follows by covering $B_1$ by a finite number of balls $B_{\frac{\sigma}{10}}(z)$  with  $z\in B_1$.
\end{proof}

 \subsection{Gradient estimates in weighted Morrey spaces}\label{sec:Morrey-Spaces}

\begin{proof}[\textbf{Proof of Theorem~\ref{thm:weighted-morrey}}]
Let $\bar x\in B_{1}$ and $0<r\leq 2$. Then by  rescaling and  applying  
 Theorem~\ref{thm:main}  we obtain
\begin{align*}
\fint_{B_r(\bar x)}|\nabla u|^q  \, dw
&\leq  C\Bigg[  \Big(\fint_{B_{2 r}(\bar x)} |\nabla u|^p dx\Big)^{\frac{q}{p}} + 
\fint_{B_{\frac32 r}(\bar x)}\M_{B_{2 r }(\bar x)}(|\F|^{p'})^{\frac{q}{p}} \, dw \Bigg].
\end{align*}
Since  $B_{\frac32 r}(\bar x)\subset B_4$ and  $B_{2 r}(\bar x)\subset B_5$, it then follows that
\begin{align}\label{eq:localized-est}
\varphi(B_r(\bar x)) \fint_{B_r(\bar x)}|\nabla u|^q  \, dw
\leq  C\Bigg[ \varphi(B_r(\bar x)) \Big( \fint_{B_{2 r}(\bar x)} |\nabla u|^p dx\Big)^{\frac{q}{p}}  +
\frac{\varphi(B_r(\bar x))}{ w(B_{\frac32 r}(\bar x))}  \int_{B_{\frac32 r}(\bar x)\cap B_{4}}\M_{B_{5}}(|\F|^{p'})^{\frac{q}{p}} 
\, dw \Bigg].
\end{align}
We next estimate the first term in the above right hand side. For this, let  $\e\in (0,n)$ to be determined later and  use the trick in  \cite[Page~2506]{MP2} to    write 
\begin{align*}
\fint_{B_{2 r}(\bar x)} |\nabla u|^p dx
&=\omega_n^{-1}
(2 r)^{-\e}\int_{B_{2 r}(\bar x)}|\nabla u|^p \bar w \, dx
\leq \omega_n^{-1} 
(2 r)^{-\e}\int_{ B_{5}}|\nabla u|^p \bar w \, dx
\end{align*}
with $\omega_n :=|B_1|$ and  $\bar w$  being  the weight defined by 
\[
\bar w(x) := \min{\{|x - \bar x|^{-n+ \e  }, (2 r)^{-n +\e }\} }.
\]
As $\bar w\in A_t$  with 
$[\bar w]_{A_t} \leq C(t,\e,n)$ for any $1<t<\infty$ (see \cite[Lemma~3.2]{MP2}), we can apply Theorem~\ref{thm:main} with $q=p$  to estimate the above last integral.
As a consequence, we obtain 
\begin{align}\label{Mo2}
 \fint_{B_{2 r}(\bar x)} |\nabla u|^p dx
&\leq  C(2 r)^{-\e}\Bigg( \bar w(B_{5} )\, \|\nabla u\|_{L^p(B_{10})}^p  +  \int_{B_{\frac{15}{2}}} \M_{B_{10}}(|\F|^{p'}) 
\bar w\, dx \Bigg)\nonumber\\
&\leq  C (2 r)^{-\e}\Bigg( \|\nabla u\|_{L^p(B_{10})}^p  +  \int_{B_{\frac{15}{2}}} \M_{B_{10}}(|\F|^{p'}) 
\bar w\, dx \Bigg)
\end{align}
with $C>0$  
 depending only on  $p$, $n$, $\omega$, $\Lambda$, $M_0$, and $\e$.
Notice that to obtain the last inequality we have used the fact
\[
\bar w(B_{5}) \leq \int_{B_6(\bar x)} |x - \bar x|^{-n+\e } \, dx
=\omega_n \int_0^6 t^{\e  -1 } dt
=\frac{ \omega_n }{\e} \, 6^{\e}. 
\]
To bound  the last integral in \eqref{Mo2}, we employ    Fubini's theorem to get
 \begin{align*}
 \int_{B_{\frac{15}{2}}} \M_{B_{10}}(|\F|^{p'})   \bar w\, dx 
 &=\int_0^\infty \int_{\{B_{\frac{15}{2}}: \bar w(x)>t\}} \M_{B_{10}}(|\F|^{p'}) \, dx dt
 \leq \int_0^{ (2 r)^{-n +\e }}   \int_{B_{t^{(-n +\e)^{-1}}}(\bar x)\cap B_{\frac{15}{2}}}\M_{B_{10}}(|\F|^{p'})  \, dx dt\\
 &\leq \int_0^{ 15^{-n +\e }}   \int_{B_{\frac{15}{2}}}\M_{B_{10}}(|\F|^{p'})  \, dx dt
 +  \int_{15^{-n +\e }}^{ (2 r)^{-n +\e }}   \int_{B_{t^{(-n +\e)^{-1}}}(\bar x)\cap B_{\frac{15}{2}}}\M_{B_{10}}(|\F|^{p'})  \, dx dt.
  \end{align*}
Since $B_{\frac{15}{2}}\subset B_{\frac{17}{2}}(\bar x)$, we have $B_{\frac{15}{2}}= B_{\frac{17}{2}}(\bar x)\cap B_{\frac{15}{2}}$ and  then deduce that
    \begin{align*}
 \int_{B_{\frac{15}{2}}} \M_{B_{10}}(|\F|^{p'})   \bar w\, dx 
 &\leq  C \| \M_{B_{10}}(|\F|^{p'})\|_{\calM^{1,\varphi^{\frac{p}{q}}}(B_{\frac{15}{2}})} 
 \left[ \varphi(B_{\frac{17}{2}}(\bar x))^{\frac{-p}{q}}  +  \int_{15^{-n +\e }}^{ (2 r)^{-n +\e }} 
t^{\frac{n}{-n +\e}} \varphi(B_{t^{(-n +\e)^{-1}}}(\bar x))^{\frac{-p}{q}}  dt\right],
  \end{align*}
  where we recall that $\calM^{1,\varphi}(U)$ denotes  the  Morrey space $\calM^{1,\varphi}_w(U)$ with $w=1$.  As $\varphi
  \in \mathcal B_+$
  by the assumption and $\{\mathcal B_\alpha\}$ is decreasing in $\alpha$, there exists $\alpha\in (0,n)$ such that  $\varphi\in
  \mathcal B_\alpha$.   It then follows if $\e< \alpha p/q$ that 
  \begin{align*}
 \int_{B_{\frac{15}{2}}} \M_{B_{10}}(|\F|^{p'})   \bar w\, dx 
 &\leq   C (2 r)^{\alpha \frac{p}{q}} \varphi(B_{2 r}(\bar x))^{\frac{-p}{q}} \| \M_{B_{10}}(|\F|^{p'})\|_{\calM^{1,\varphi^{\frac{p}{q}}}(B_{\frac{15}{2}})} 
 \Big[ 1  + 
  \int_0^{ (2 r)^{-n +\e }} 
t^{\frac{n-\alpha \frac{p}{q}}{-n +\e}}   dt\Big]\\
 &\leq   C    (2 r)^{\e} \varphi(B_r(\bar x))^{\frac{-p}{q}}
 \| \M_{B_{10}}(|\F|^{p'})\|_{\calM^{1,\varphi^{\frac{p}{q}}}(B_{\frac{15}{2}})}.
   \end{align*}
 %  with $C>0$ depending only on $\e,\, \alpha,\, p,\, q$, and $\varphi$.
 Combining this with \eqref{Mo2}, we arrive at:
\begin{align*}
 \fint_{B_{2 r}(\bar x)} |\nabla u|^p dx
\leq   C\Bigg(r^{-\e} \|\nabla u\|_{L^p(B_{10})}^p+ \varphi(B_r(\bar x))^{\frac{-p}{q}}
\| \M_{B_{10}}(|\F|^{p'})\|_{\calM^{1,\varphi^{\frac{p}{q}}}(B_{\frac{15}{2}})} \Bigg).
\end{align*}
Therefore, we infer from \eqref{eq:localized-est} and the fact $\varphi\in \mathcal B_\alpha$  that 
\begin{align*}
\varphi(B_r(\bar x)) 
\fint_{B_r(\bar x)}|\nabla u|^q  \, dw
&\leq  C_\e\Bigg[\varphi(B_r(\bar x)) r^{\frac{-\e q}{p}}\|\nabla u\|_{L^p(B_{10})}^q
+\| \M_{B_{10}}(|\F|^{p'})\|_{\calM^{1,\varphi^{\frac{p}{q}}}(B_{\frac{15}{2}})}^{\frac{q}{p}}
+
\|\M_{B_{5}} (|\F|^{p'})\|_{\calM^{\frac{q}{p},\varphi}_w(B_{4})}^{\frac{q}{p}}\Bigg]\\
&\leq  C_\e\Bigg[\varphi(B_2(\bar x)) r^{\alpha-\frac{\e q}{p}}\|\nabla u\|_{L^p(B_{10})}^q
+\| \M_{B_{10}}(|\F|^{p'})\|_{\calM^{1,\varphi^{\frac{p}{q}}}(B_{\frac{15}{2}})}^{\frac{q}{p}}
+
\|\M_{B_{5}} (|\F|^{p'})\|_{\calM^{\frac{q}{p},\varphi}_w(B_{4})}^{\frac{q}{p}}\Bigg]
\end{align*}
for all $\bar x\in B_{1}$ and $0<r\leq 2$.  By taking $\e = \frac{\alpha}{2}\frac{p}{q}$
and as $\sup_{\bar x\in B_1} \varphi( B_2(\bar x))<\infty$, this gives estimate \eqref{weighted-morrey-est}.
   \end{proof}
  
  \begin{proof}[\textbf{Proof of Theorem~\ref{cor:weighted-morrey}}]
Since  $v\in A_{\frac{q}{p}}$, 
  Lemma~\ref{weight:basic-pro} gives 
    $ v^{1-(\frac{q}{p})'} \in A_{(\frac{q}{p})'}\subset A_\infty$. Thus our assumptions imply that condition  ({\bf B}) 
  in  Corollary~\ref{cor:two-weight-case} is satisfied.
Also as $\varphi \in \mathcal B_0$, it is clear that \eqref{cond:vwvarphi}  implies \eqref{wmu1-vmu2-sufficient}. Indeed, for any 
$y\in \R^n$ and any 
$s\geq 2r>0$ we have from \eqref{cond:vwvarphi} and $\varphi \in \mathcal B_0$ that
\[
  \frac{v(B_s(y))}{w(B_s(y))} \frac{1}{\phi(B_s(y))} 
  \leq C_* \, \frac{1}{\varphi(B_{\frac{s}{2}}(y))} \leq C_* C \, \frac{1}{\varphi(B_r(y))}
\]
yielding \eqref{wmu1-vmu2-sufficient}.
Moreover, by  \cite[Theorem~9.3.3]{Gra} there exist $s\in (1,\infty)$ and $C>0$ depending only on $n$ and $[w]_{A_\infty}$ such that $[w]_{A_s}\leq C$.
 Therefore, it follows  from Theorem~\ref{thm:weighted-morrey} and  Corollary~\ref{cor:two-weight-case} that
\begin{equation}\label{eq:1.3+3.4}
\| \nabla u \|_{\calM^{q,\varphi}_w(B_1)}  \leq    C\Bigg(
\|\nabla u\|_{L^p(B_{10})}
+ \| \M_{B_{10}}(|\F|^{p'})\|_{\calM^{1, \varphi^{\frac{ p}{q}}}(B_{10})}^{\frac{1}{p}}
+  \| |\F|^{\frac{1}{p-1}}\|_{\calM^{q,\phi}_v(B_{10})}\Bigg)
\end{equation}
with $C>0$ depending  only on  $q$, $p$, $n$, $\omega$, $\Lambda$,   $M_0$,  $\varphi$, $C_*$, $[w]_{A_\infty}$, $[v]_{A_{\frac{q}{p}}}$,  and $[w, v^{1-(\frac{q}{p})'}]_{A_{\frac{q}{p}}}$.
Thus it remains to estimate the middle term on the right hand side of \eqref{eq:1.3+3.4}. Let $l:=q/p>1$. Then 
for  any nonnegative function $g\in L^1(B_{10})$,  we obtain  from   H\"older inequality and assumption  \eqref{cond:wv}
   %and the assumption  $w\in A_{\frac{q}{q}}$ 
   that
  \begin{align*}
 \frac{\varphi(B_R(\bar x))}{|B_R(\bar x)|^l} \Big(\int_{B_R(\bar x)\cap B_{10}} g \, dx\Big)^l
&\leq \frac{\varphi(B_R(\bar x))}{|B_R(\bar x)|^l} \Big( \int_{B_R(\bar x)\cap B_{10}} g^l  v\, dx\Big) 
 \Big(\int_{B_R(\bar x)} v^{1 - l'} \Big)^{l-1}\\
 &\leq  [w, v^{1-l'}]_{A_l}  \frac{\varphi(B_R(\bar x))}{w\big( B_R(\bar x)\big)}  \int_{B_R(\bar x)\cap B_{10}} g^l  v\, dx
 \leq   [w, v^{1-l'}]_{A_l} \|g\|_{\calM^{l,\hat\varphi}_v(B_{10})}^l\nonumber
  \end{align*}
  for all $\bar x\in B_{10}$ and all $0<R\leq 20$, where
  \[
  \hat\varphi(B) :=  \frac{v(B)}{w(B)} \varphi(B).
  \]
  Hence we infer  that
 \begin{equation}\label{eq:est-in-terms-v}
   \| \M_{B_{10}}(|\F|^{p'})\|_{\calM^{1, \varphi^{\frac{1}{l}}}(B_{10})}\leq 
    [w, v^{1-l'}]_{A_l}^{\frac{1}{l}} \,\, \| \M_{B_{10}}(|\F|^{p'})\|_{\calM^{l,\hat\varphi}_v(B_{10})}.
  \end{equation}
Using  $\phi\in \mathcal B_0$, condition \eqref{cond:vwvarphi}, and the doubling property of $w$ due to Lemma~\ref{weight:basic-pro}, we have
\begin{align*}
\sup_{s\geq 2r}  \frac{1}{\phi(B_s(y))} \leq
C \frac{1}{\phi(B_{2r}(y))}\leq C C_* \frac{w(B_{2r}(y))}{v(B_{2r}(y))}
\frac{1}{\varphi(B_r(y))}\leq C' \frac{w(B_r(y))}{v(B_r(y))}
\frac{1}{\varphi(B_r(y))} =C'\frac{1}{\hat\varphi(B_r(y))}
\end{align*} 
for all $y\in \R^n$ and $r>0$.    
 Thus as $v\in A_l$ we can use the strong type estimate for   the Hardy--Littlewood   maximal function 
 given by Corollary~\ref{cor:one-weight-case} to estimate the right hand side of \eqref{eq:est-in-terms-v}. As a result, we get
 \begin{equation}\label{eq:option-1}
   \| \M_{B_{10}}(|\F|^{p'})\|_{\calM^{1, \varphi^{\frac{1}{l}}}(B_{10})}\leq C \| |\F|^{p'}\|_{\calM^{l,\phi}_v(B_{10})}
    = C \| |\F|^{\frac{1}{p-1}}\|_{\calM^{q,\phi}_v(B_{10})}^p.
  \end{equation}
 This and \eqref{eq:1.3+3.4} yield  desired estimate
\eqref{weighted-morrey-final}.
\end{proof}

We close the paper by noting that 
%\begin{remark}\label{rmk:variation} We note that 
 if  $w\in A_{\frac{q}{p}}$, then the assumption  $v\in A_{\frac{q}{p}}$ in Theorem~\ref{cor:weighted-morrey} can be disposed if  condition \eqref{cond:wv} is strengthened by assuming that there exists $r>1$ such that
\begin{equation*}
 \sup_{B}{ \Big(\fint_{B} w\, dx\Big) \Big(\fint_{B} v^{r[1-(\frac{q}{p})']}\, dx\Big)^{\frac{q-p}{pr }}  } 
  <\infty. 
 \end{equation*}
This  follows from the above proof of  Theorem~\ref{cor:weighted-morrey} and the fact that  condition  ({\bf A})  in  Corollary~\ref{cor:two-weight-case} is satisfied in this case. Indeed, having  ({\bf A})  allows us to use  Corollary~\ref{cor:two-weight-case} to obtain \eqref{eq:1.3+3.4} 
and it remains to show  \eqref{eq:option-1}. But  for any nonnegative function $g\in L^1(B_{10})$,  we have from   H\"older inequality  and $w\in A_l$    that
  \begin{align*}
 \frac{\varphi(B_R(\bar x))}{|B_R(\bar x)|^l} \Big(\int_{B_R(\bar x)\cap B_{10}} g \, dx\Big)^l
&\leq \frac{\varphi(B_R(\bar x))}{|B_R(\bar x)|^l} \Big( \int_{B_R(\bar x)\cap B_{10}} g^l  w\, dx\Big) 
 \Big(\int_{B_R(\bar x)} w^{1-l'} \Big)^{l-1}\\
 &\leq  [w]_{A_l}  \frac{\varphi(B_R(\bar x))}{w\big( B_R(\bar x)\big)}  \int_{B_R(\bar x)\cap B_{10}} g^l  w\, dx
 \leq   [w]_{A_l} \|g\|_{\calM^{l,\varphi}_w(B_{10})}^l\nonumber
  \end{align*}
  for all $\bar x\in B_{10}$ and all $0<R\leq 20$.
This observation  together with  Corollary~\ref{cor:two-weight-case} gives
 \begin{equation*}
   \| \M_{B_{10}}(|\F|^{p'})\|_{\calM^{1, \varphi^{\frac{ 1}{l}}}(B_{10})}\leq 
    [w]_{A_l}^{\frac{1}{l}} \,\, \| \M_{B_{10}}(|\F|^{p'})\|_{\calM^{l,\varphi}_w(B_{10})}\leq 
    C \| |\F|^{p'}\|_{\calM^{l,\phi}_v(B_{10})}
    = C \| |\F|^{\frac{1}{p-1}}\|_{\calM^{q,\phi}_v(B_{10})}^p
  \end{equation*}
  which is precisely  \eqref{eq:option-1}.
  As a consequence, we still arrive at  conclusion \eqref{weighted-morrey-final}.\\
  
\noindent {\bf Acknowledgement} The authors would like to thank the anonymous referee for valuable  suggestions which improve the presentation of the paper.

    \Addresses
  \end{document}